\documentclass[a4paper,reqno]{amsart}

\usepackage[T1]{fontenc}
\usepackage[utf8x]{inputenc}
\usepackage[english]{babel}
\usepackage{yfonts}
\usepackage{dsfont}
\usepackage{amscd,amssymb,amsmath,amsthm,amsfonts}
\usepackage{mathtools}
\usepackage{accents}
\usepackage{tensor}
\usepackage{graphicx}
\usepackage{tikz}
\usetikzlibrary{calc,matrix,arrows,decorations.pathmorphing}
\usepackage{calc}
\usepackage[cal=boondox,scr=boondoxo]{mathalfa}
\usepackage{enumitem}
\usepackage{marginnote}
\usepackage{booktabs}
\usepackage{scalerel}[2016/12/29]
\usepackage{url}
\usepackage{contour}
\usepackage{ulem}
\usepackage{accents}
\usepackage{placeins}

\usepackage{hyperref}
\hypersetup{colorlinks=true,pageanchor=false,
linkcolor=blue,citecolor=red,urlcolor=red}
\usepackage[all]{hypcap}

\allowdisplaybreaks

\newtheorem{theorem}{Theorem}[section]

\newtheorem{proposition}[theorem]{Proposition}
\newtheorem{lemma}[theorem]{Lemma}

\theoremstyle{definition}
\newtheorem{definition}[theorem]{Definition}

\theoremstyle{remark}

\newcommand{\Z}{\mathbb{Z}}

\newcommand{\R}{\mathbb{R}}

\newcommand{\rmt}{\mathrm{t}}

\newcommand{\rmF}{\mathrm{F}}

\newcommand{\rmI}{\mathrm{I}}

\newcommand{\rmL}{\mathrm{L}}

\newcommand{\rmV}{\mathrm{V}}

\newcommand{\bbD}{\mathbb{D}}

\newcommand{\bbI}{\mathbb{I}}

\newcommand{\bbM}{\mathbb{M}}

\newcommand{\bbP}{\mathbb{P}}
\newcommand{\bbS}{\mathbb{S}}

\newcommand{\bbW}{\mathbb{W}}

\DeclareRobustCommand{\bbGamma}{\mathbin{\text{\includegraphics[height=\heightof{$\mathbf{\Gamma}$}]{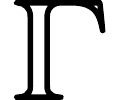}}}}

\DeclareRobustCommand{\bbSigma}{\mathbin{\text{\includegraphics[height=\heightof{$\mathbf{\Sigma}$}]{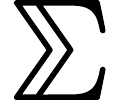}}}}

\newcommand{\bfA}{\mathbf{A}}

\newcommand{\bfC}{\mathbf{C}}

\newcommand{\bfF}{\mathbf{F}}

\newcommand{\bfepsilon}{\boldsymbol{\varepsilon}}

\newcommand{\bfmu}{\boldsymbol{\mu}}

\newcommand{\bfsigma}{\boldsymbol{\sigma}}

\newcommand{\bfchi}{\boldsymbol{\chi}}

\newcommand{\calA}{\mathcal{A}}

\newcommand{\calC}{\mathcal{C}}
\newcommand{\calD}{\mathcal{D}}

\newcommand{\calL}{\mathcal{L}}

\newcommand{\calR}{\mathcal{R}}
\newcommand{\calS}{\mathcal{S}}

\newcommand{\calV}{\mathcal{V}}

\newcommand{\bcalC}{\mathbcal{C}}

\renewcommand{\epsilon}{\varepsilon}
\renewcommand{\theta}{\vartheta}
\renewcommand{\phi}{\varphi}
\renewcommand{\Gamma}{\varGamma}
\renewcommand{\Sigma}{\varSigma}

\newcommand{\id}{\mathrm{id}}

\newcommand{\ptr}{\mathrm{ptr}}

\newcommand{\im}{\operatorname{im}}

\newcommand{\lev}{\smash{\stackrel{\leftarrow}{\mathrm{ev}}}}
\newcommand{\lcoev}{\smash{\stackrel{\longleftarrow}{\mathrm{coev}}}}
\newcommand{\rev}{\smash{\stackrel{\rightarrow}{\mathrm{ev}}}}
\newcommand{\rcoev}{\smash{\stackrel{\longrightarrow}{\mathrm{coev}}}}

\DeclareMathOperator{\Forall}{\forall}

\DeclareMathOperator{\sqtimes}{\scaleobj{0.8}{\boxtimes}}
\DeclareMathOperator{\csqtimes}{\hat{\scaleobj{0.8}{\boxtimes}}}
\DeclareMathOperator{\disjun}{\sqcup}
\DeclareMathOperator{\din}{\dot{\Rightarrow}}

\newcommand{\leqs}{\leqslant}

\newcommand{\mods}[1]{\operatorname{\mathnormal{#1}-mod}}

\newcommand{\bdots}{\reflectbox{$\ddots$}}
\newcommand{\fsl}{\mathfrak{sl}}

\newcommand{\End}{\mathrm{End}}

\newcommand{\Vect}{\mathrm{Vect}}

\newcommand{\op}{\mathrm{op}}

\DeclareRobustCommand{\one}{\mathbin{\text{\includegraphics[height=\heightof{$\mathbf{1}$}]{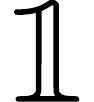}}}}

\newcommand{\bfCob}{\mathbf{Cob}}
\newcommand{\bfadCob}{\check{\mathbf{C}}\mathbf{ob}}
\newcommand{\rmadCob}{\check{\mathrm{C}}\mathrm{ob}}

\newcommand{\bfCat}{\mathbf{Cat}}
\newcommand{\coCat}{\hat{\mathbf{C}}\mathbf{at}}

\newcommand{\cmpl}{\bfC}

\newcommand{\Kar}{\mathrm{K}}
\newcommand{\Mat}{\mathrm{M}}

\newcommand{\Proj}{\mathrm{Proj}}

\newcommand{\adSk}{\check{\mathrm{S}}}
\newcommand{\adD}{\check{\mathrm{D}}}

\newcommand{\bbProj}{\mathbb{P}\mathrm{roj}}

\newcommand{\Sk}{\mathrm{S}}

\makeatletter
\newcommand{\subalign}[1]{
  \vcenter{
    \Let@ \restore@math@cr \default@tag
    \baselineskip\fontdimen10 \scriptfont\tw@
    \advance\baselineskip\fontdimen12 \scriptfont\tw@
    \lineskip\thr@@\fontdimen8 \scriptfont\thr@@
    \lineskiplimit\lineskip
    \ialign{\hfil$\m@th\scriptstyle##$&$\m@th\scriptstyle{}##$\crcr
      #1\crcr
    }
  }
}
\makeatother

\def\clap#1{\hbox to 0pt{\hss#1\hss}}

\newcommand{\coend}{\mathcal{L}} 

\newcommand{\intL}{\Lambda}

\newcommand{\pic}[2][0]{\raisebox{-0.5\height + 2.5pt + #1pt}{\includegraphics{#2.pdf}}}

\newcommand\arxiv[2]{\href{https://arXiv.org/abs/#1}{\texttt{arXiv:\allowbreak #1} #2}}
\newcommand\doi[2]{\href{https://doi.org/#1}{#2}}

\contourlength{1pt}

\DeclareRobustCommand{\myuline}[1]{
 \ifmmode \text{\uline{$\phantom{#1}$}\llap{\contour{white}{$#1$}}}
 \else \uline{\phantom{#1}}\llap{\contour{white}{#1}} \fi
}

\newcommand{\pant}{\otimes}
\newcommand{\pantbar}{\myuline{\calL \otimes \phantom{\_}}}

\begin{document}

\raggedbottom

\title{Extended TQFTs From Non-Semisimple Modular Categories}

\author[M. De Renzi]{Marco De Renzi} 
\address{Institute of Mathematics, University of Zurich, Winterthurerstrasse 190, CH-8057 Zurich, Switzerland} 
\email{marco.derenzi@math.uzh.ch}

\begin{abstract}
 We construct 3-dimensional once-Extended Topological Quantum Field Theories (ETQFTs for short) out of (possibly non-semisimple) modular categories, and we explicitly identify linear categories and functors in their image. The circle category of an ETQFT produced by our construction is equivalent to the full subcategory of projective objects of the underlying modular category. In particular, it need not be semisimple.
\end{abstract}

\maketitle
\setcounter{tocdepth}{3}

\section{Introduction}

The goal of this paper is to construct and characterize 3-dimensional Extended Topological Quantum Field Theories (ETQFTs for short) starting from an algebraic structure called a modular category. Our result provides a 2-categorical extension of the non-semisimple TQFTs that were build in \cite{DGGPR19} and further studied in \cite{DGGPR20}. It also provides a non-semisimple counterpart to the class of semisimple ETQFTs that were completely classified by Bartlett, Douglas, Schommer-Pries, and Vicary in \cite{BDSV15}, although classification questions are not addressed in this work. The result is very closely related to the one that was obtained in \cite{D17} starting from relative modular categories. However, since the algebraic ingredients considered here are less intricate, this version of the construction is free from much of the technical challenges that characterize the Costantino--Geer--Patureau (CGP for short) theory. This allows for an illustration of the main ideas at play in a more straightforward manner.

Some care is needed in order to explain precisely our terminology, as several distinct notions carry the same name in the literature. The term \textit{modular category}, for instance, is employed here in the not necessarily semisimple sense of Lyubashenko \cite{L94}, whose definition differs from the original one of Turaev \cite{T94} in that the semisimplicity assumption is dropped. In other words, to us a modular category $\calC$ will always mean a finite ribbon category whose braiding satisfies an appropriate non-degeneracy condition, see Section~\ref{S:modular_categories}, but which is not required to be semisimple. Another term that can have several concurring definitions is the term \textit{ETQFT}. It is used here in the once-extended, 2-categorical sense whose first traces can be found in the work of Freed \cite{F92} and Lawrence \cite{L93}, and which was systematically treated by Schommer-Pries \cite{S11} and by Bartlett, Douglas, Schommer-Pries, and Vicary \cite{BDSV15}. Following this approach, an ETQFT can be defined as a symmetric monoidal 2-functor with source a 2-category of cobordisms, and with several possible targets provided by linear 2-categories of some sort. Remark that here the term 2\textit{-category} has to be understood in the fully weak sense of Bénabou \cite{B67} (who used the term \textit{bicategory} instead). In our case, the target for ETQFTs will be provided by the 2-category $\coCat_\Bbbk$ of (additive and idempotent) complete linear categories, functors, and natural transformations, see Definition~\ref{D:sym_mon_2-cat_of_co_lin_cat}, while the definition of the source 2-category $\bfadCob_\calC$ will depend on the modular category $\calC$. Roughly speaking, in $\bfadCob_\calC$ objects are closed 1-dimensional manifolds, 1-morphisms are 2-dimensional cobordisms decorated with marked points labeled by objects of $\calC$, and 2-morphisms are 3-dimensional cobordisms with corners decorated with so-called \textit{bichrome graphs} labeled by objects and morphisms of $\calC$. These decorations, whose definition is recalled in Section~\ref{S:3-manifold_invariant}, need to satisfy an admissibility condition that consists in requiring the presence of a projective object of $\calC$ among the labels of decorations of connected components of cobordisms which are disjoint from the incoming boundary, see Definition~\ref{D:sym_mon_2-cat_of_adm_cob}. Our main result can then be stated as follows, see Theorem~\ref{T:symmetric_monoidality} and Proposition~\ref{P:univ_lin_cat}.

\begin{theorem}\label{T:main}
 If $\calC$ is a modular category over an algebraically closed field $\Bbbk$, then there exists an ETQFT $\hat{\bfA}_\calC : \bfadCob_\calC \to \coCat_\Bbbk$ whose circle category satisfies
 \[
  \hat{\bfA}_\calC(\bbS^1) \cong \Proj(\calC).
 \]
\end{theorem}

In particular, when $\calC$ is semisimple (and thus modular in the sense of Turaev), the construction recovers the corresponding ETQFT appearing in the classification of \cite{BDSV15}. The statement of Theorem~\ref{T:main} might be surprising, because the full subcategory $\Proj(\calC)$ of projective objects of $\calC$ is not in general a monoidal category (it is if and only if $\calC$ is semisimple). This has to do with the admissibility condition on morphisms of $\bfadCob_\calC$. Indeed, connected 2-dimensional cobordisms with empty incoming boundary need to carry projective labels, so that in particular the disc $D^2$, when interpreted as a cobordism from $\varnothing$ to $S^1$, cannot be equipped with a trivial decoration. Since it is this morphism that would determine the unit for the tensor product on $\hat{\bfA}_\calC(\bbS^1)$, its absence from $\bfadCob_\calC$ is consistent with the absence of the tensor unit of $\calC$ from $\Proj(\calC)$ (whenever $\calC$ is non-semisimple). Under this equivalence, disc functors can be interpreted in terms of morphism spaces of $\calC$, while pant functors in terms of the tensor product and the coend of $\calC$. It would be interesting to understand whether the admissibility condition for morphisms of $\bfadCob_\calC$ can be relaxed to allow for closed surfaces with empty decorations, and if the ETQFT $\hat{\bfA}_\calC$ can be extended over this larger 2-category of cobordisms, thus yielding a circle category which is equivalent to the whole modular category $\calC$.

\subsection{Survey of previous literature}

Motivated by the seminal work of Witten \cite{W88}, a set of axioms for the notion of TQFT was first proposed by Atiyah \cite{A88}, who was inspired by a similar definition for the concept of Conformal Field Theory (CFT for short) due to Segal \cite{S88}.
In modern terms, a TQFT is defined as a symmetric monoidal functor from a category of cobordisms to the category of vector spaces. Following his groundbreaking work with Reshetikhin \cite{RT91}, Turaev explained in \cite{T94} how to construct a TQFT starting from a semisimple modular category, a notion he introduced precisely for this purpose. We point out that these algebraic structures are typically called modular categories, and that we specify that we are talking about the semisimple version of the concept here in order to distinguish it from modular categories as defined by Lyubashenko in \cite{L94}.

The question of extending TQFTs to lower dimensional manifolds is an old one. Surfaces with boundary were first discussed using the \textit{modular functor} approach by Walker \cite{W91}, and later by Turaev \cite{T94}, Blanchet, Habegger, Masbaum, and Vogel \cite{BHMV95}, and Bakalov and Kirillov \cite{BK01}, among others. For our construction, we will rather adopt a higher categorical approach which encompasses modular functors, following ideas first considered by Freed \cite{F92} and Lawrence \cite{L93}, and greatly promoted by Baez and Dolan \cite{BD95}, whose \textit{cobordism hypothesis} was famously proved by Lurie \cite{L09}. In particular, the 2-categorical treatment we will follow here is closest to the one of Bartlett, Douglas, Schommer-Pries, and Vicary \cite{BDSV15}, who completely classified 3-dimensional ETQFTs in terms of semisimple modular categories. For their classification, they considered several flavors of 2-categories of cobordisms, all of which are rigid. Since the 2-categories of cobordisms considered here as sources for ETQFTs do not meet this condition in general, we will step outside of the semisimple framework.

The first attempt at constructing ETQFTs out of non-semisimple algebraic ingredients is due to Kerler and Lyubashenko \cite{KL01}, who used double-categories instead of 2-categories. However, their result builds on the 3-manifold invariants defined by Hennings \cite{H96} and Lyubashenko \cite{L94}, whose properties determine deep obstructions to monoidality of functorial extensions in the sense of Atiyah, that is, with respect to disjoint union. Indeed, the invariant associated with a non-semisimple modular category vanishes against every closed 3-manifold whose first Betti number is strictly positive. This forces Kerler and Lyubashenko to work exclusively with connected surfaces and 3-manifolds, and to consider a weaker notion of monoidality based on connected sum instead of disjoint union. 

One way to overcome these problems is to use the theory of \textit{modified traces} developed by Geer, Kujawa, and Patureau \cite{GKP10}. These techniques were first used to construct invariants of closed 3-manifolds by Costantino, Geer, and Patureau in \cite{CGP12}. Such invariants were then extended to graded TQFTs defined over non-rigid categories of cobordisms in \cite{BCGP14}, but only in a special case arising from the representation theory of the unrolled quantum group of $\fsl_2$. The construction of ETQFTs for the CGP theory in a general setting was achieved in \cite{D17} using \textit{relative modular categories}, a different non-semisimple generalization of Turaev's notion of semisimple modular category. Since these algebraic ingredients are rather complicated, being modeled on the representation theory of infinite-dimensional unrolled quantum groups, the resulting ETQFTs are defined over non-rigid 2-categories of cobordisms decorated with cohomology classes, and take values in 2-categories of complete graded linear categories. 

The use of modified traces was subsequently integrated into Hennings' construction first \cite{DGP17}, and into Lyubashenko's construction later \cite{DGGPR19}, yielding TQFTs which extend Lyubashenko's mapping class group representations to symmetric monoidal functors defined over non-rigid categories of cobordisms \cite{DGGPR20}. A discussion of the relation between this approach and the CGP one, in the case of quantum groups, can be found in \cite{DGP18}. Our goal in this paper is to apply the construction of \cite{D17} in Lyubashenko's setting, where cohomology classes and graded extensions are not needed.

\subsection{Outline of the construction}

Starting from a modular category $\calC$ over an algebraically closed field $\Bbbk$, we can consider the renormalized Lyubashenko invariant $\rmL'_\calC$ constructed in \cite{DGGPR19}. This is a topological invariant of closed connected 3-manifolds decorated with \textit{admissible bichrome graphs}, which are special ribbon graphs whose edges can be either \textit{red}, and unlabeled, or \textit{blue}, and labeled by objects of $\calC$. The admissibility condition requires the set of labels of blue edges to contain at least a projective object of $\calC$, as explained in Section~\ref{S:3-manifold_invariant}. The universal construction of \cite{BHMV95} can then be used to extend the invariant $\rmL'_\calC$ to a pair of TQFTs $\rmV_\calC : \rmadCob_\calC \to \Vect_\Bbbk$ and $\rmV'_\calC : (\rmadCob_\calC)^\op \to \Vect_\Bbbk$, where $\rmadCob_\calC$ is the \textit{category of admissible cobordisms} recalled in Section~\ref{S:TQFT}. Very loosely speaking, the procedure is the following: for each object $\bbSigma$ of $\rmadCob_\calC$, we should first consider linearized versions of the sets of morphisms $\rmadCob_\calC(\varnothing,\bbSigma)$ and $\rmadCob_\calC(\bbSigma,\varnothing)$, given by vector spaces denoted $\calV(\bbSigma)$ and $\calV'(\bbSigma)$ respectively. Then, we should consider the bilinear map
\[
 \langle \_,\_ \rangle_{\bbSigma} : \calV'(\bbSigma) \times \calV(\bbSigma) \to \Bbbk
\]
induced by the invariant $\rmL'_\calC$, see Section~\ref{S:TQFT}. The vector spaces $\rmV_\calC(\bbSigma)$ and $\rmV'_\calC(\bbSigma)$ are defined as the quotients 
\begin{align*}
 \rmV_\calC(\bbSigma) &:= \calV(\bbSigma) / \calV'(\bbSigma)^\perp, 
 &
 \rmV'_\calC(\bbSigma) &:= \calV'(\bbSigma) / \calV(\bbSigma)^\perp
\end{align*}
with respect to the annihilators associated with the pairing $\langle \_,\_ \rangle_{\bbSigma}$. We point out that symmetric monoidality of $\rmV_\calC$ and $\rmV'_\calC$ is not at all obvious, but it is rather the result of a careful analysis, and a consequence of the modularity condition of $\calC$.

The idea for defining 2-functorial extensions of $\rmL'_\calC$ is essentially the same, only translated one categorical level higher. The early germs of this extended version of the universal construction date back to \cite{BHMV95}, although the approach first appeared in its full 2-categorical form in \cite{D17}. The so-called \textit{extended universal construction} can be used to define, starting from the invariant $\rmL'_\calC$, a pair of ETQFTs $\bfA_\calC : \bfadCob_\calC \to \coCat_\Bbbk$ and $\bfA'_\calC : (\bfadCob_\calC)^\op \to \coCat_\Bbbk$. Again, very loosely speaking, the procedure is the following: for each object $\bbGamma$ of $\bfadCob_\calC$, we should first consider linearized versions of the categories of morphisms $\bfadCob_\calC(\varnothing,\bbGamma)$ and $\bfadCob_\calC(\bbGamma,\varnothing)$, given by linear categories denoted $\calA(\bbGamma)$ and $\calA'(\bbGamma)$ respectively. Then, we should consider the bilinear functor
\[
 \langle \_,\_ \rangle_{\bbGamma} : \calA'(\bbGamma) \times \calA(\bbGamma) \to \Vect_\Bbbk
\]
induced by the TQFT $\rmV_\calC$, see Section~\ref{S:extended_universal_construction}. The linear categories $\bfA_\calC(\bbGamma)$ and $\bfA'_\calC(\bbGamma)$ are then defined as the quotients 
\begin{align*}
 \bfA_\calC(\bbSigma) &:= \calA(\bbGamma) / \calA'(\bbGamma)^\perp, 
 &
 \bfA'_\calC(\bbSigma) &:= \calA'(\bbGamma) / \calA(\bbGamma)^\perp
\end{align*}
with respect to the annihilators associated with the pairing $\langle \_,\_ \rangle_{\bbGamma}$, where a precise sense to this notation is given in Section~\ref{S:extended_universal_construction}. Just like in the previous case, some work is required in order to establish the symmetric monoidality of these 2-functors. This takes up Sections~\ref{S:combinatorial_topological_properties} and \ref{S:monoidality}.

In Section~\ref{S:identification_of_image}, the image of $\bfA_\calC$ is characterized in terms of the modular category $\calC$. More precisely, Proposition~\ref{P:univ_lin_cat} states that the additive and idempotent completion of the circle category $\bfA_\calC(\bbS^1)$ is equivalent to the full subcategory $\Proj(\calC)$ of projective objects of $\calC$. Under this equivalence, functors associated with generating 1-morphisms of $\bfadCob_\calC$ can be explicitly identified:
\begin{enumerate}
 \item A disc $\bbD^2_{(+,V)} : \varnothing \to \bbS^1$, decorated with a single positive marked point labeled by $V \in \Proj(\calC)$, is translated to the functor
  \[
   V : \Bbbk \to \Proj(\calC)
  \]
  sending the unique object of the linear category $\Bbbk$ to $V$, see Proposition~\ref{P:unit} (remark that $V$ is required to be projective, as a consequence of the admissibility condition);
 \item A disc $\overline{\bbD^2_{(-,V)}} : \bbS^1 \to \varnothing$ with opposite orientation, and decorated with a single negative marked point labeled by $V \in \calC$, is translated to the functor
  \[
   \calC(V,\_) : \Proj(\calC) \to \Vect_\Bbbk
  \]
  sending every object $V'$ of $\Proj(\calC)$ to the vector space of morphisms $\calC(V,V')$, see Proposition~\ref{P:counit} (remark that this time $V$ can be arbitrary);
 \item A pair of pants $\bbP^2 : \bbS^1 \disjun \bbS^1 \to \bbS^1$ is translated to the functor
  \[
   \pant : \Proj(\calC) \sqtimes \Proj(\calC) \to \Proj(\calC)
  \]
  sending every object $(V,V')$ of $\Proj(\calC) \sqtimes \Proj(\calC)$ to $V \otimes V'$, see Proposition~\ref{P:product};
 \item A pair of pants $\overline{\bbP^2} : \bbS^1 \to \bbS^1 \disjun \bbS^1$ with opposite orientation is translated to the functor
  \[
   \pantbar : \Proj(\calC) \to \Proj(\calC) \sqtimes \Proj(\calC)
  \]
  sending every object $V$ of $\Proj(\calC)$ to the splitting, in $\Proj(\calC) \sqtimes \Proj(\calC)$, of the coend $\coend \otimes V$, see Propositions~\ref{P:uccidetemi}-\ref{P:coproduct_critical}.
\end{enumerate}

We highlight point $(iv)$, which can be understood, using Proposition~\ref{P:coproduct_critical}, as a non-semisimple version of the standard property of modular functors corresponding to the \textit{gluing axiom}~(2.2) in \cite[Section~2]{W91}, to the \textit{splitting axiom}~(1.6.3) in \cite[Section~V.1.6]{T94}, to the \textit{colored splitting theorem} in \cite[Section~1.14]{BHMV95}, or to the \textit{gluing isomorphism}~$(iv)$ in \cite[Definition~5.1.13]{BK01}. Notice however that, in this setting, the property is parametric, meaning it enables us to cut and glue decorated surfaces along simple closed curves only in the proximity of a marked point labeled by an object $V \in \Proj(\calC)$, see Figure~\ref{F:I_times_D_2}.

\subsection*{Acknowledgments}

The author would like to thank Azat Gainutdinov for several helpful discussions.

\subsection*{Conventions}

Every manifold in this paper is assumed to be oriented, and every diffeomorphism orientation-preserving. We also fix an algebraically closed field $\Bbbk$, and we use the term linear as a shorthand for $\Bbbk$-linear.

\section{TQFTs from modular categories}

In this section, we recall the definition of the 3-dimensional TQFT $\rmV_\calC$ constructed in \cite{DGGPR19} out of a modular category $\calC$.

\subsection{Modular categories}\label{S:modular_categories}

First of all, we need to recall the definition of the basic algebraic ingredients for the construction. We start from Lyubashenko's notion of modular category, for which there exist several equivalent definitions, see for instance \cite[Section~6]{BD20}. The most concise one is the following: a \textit{modular category} is a finite ribbon category over $\Bbbk$ whose transparent objects are trivial, meaning they are isomorphic to direct sums of copies of the tensor unit.

Let us explain the terminology. A \textit{finite category} is a linear category $\calC$ over $\Bbbk$ which is equivalent to the category $\mods{A}$ of finite dimensional left $A$-modules for a finite dimensional algebra $A$ over $\Bbbk$, see \cite[Definition~1.8.6]{EGNO15} for an equivalent definition. A \textit{ribbon structure} on a linear category $\calC$ is given by:
\begin{enumerate}
 \item a tensor product $\otimes : \calC \times \calC \to \calC$;
 \item a tensor unit $\one \in \calC$;
 \item left and right evaluation and coevaluation morphisms $\lev_X : X^* \otimes X \to \one$, $\lcoev_X : \one \to X \otimes X^*$, $\rev_X : X \otimes X^* \to \one$, $\rcoev_X : \one \to X^* \otimes X$ for every $X \in \calC$;
 \item braiding morphisms $c_{X,Y} : X \otimes Y \to Y \otimes X$ for all $X,Y \in \calC$;
 \item twist morphisms $\theta_X : X \to X$ for every $X \in \calC$.
\end{enumerate}
This structure is required to satisfy several axioms, which can be found in \cite[Sections~2.1, 2.10, 8.1, 8.10]{EGNO15}. An object $X \in \calC$ of a ribbon category is said to be \textit{transparent} if it belongs to the Müger center $M(\calC)$ of $\calC$, which means its double braiding satisfies $c_{Y,X} \circ c_{X,Y} = \id_{X \otimes Y}$ for every object $Y \in \calC$.

A modular category $\calC$ always admits a \textit{coend}
\[
 \coend := \int^{X \in \calC} X^* \otimes X \in \calC,
\]
which is, by definition, the universal dinatural transformation with source
\begin{align*}
  (\_^* \otimes \_) : \calC^\op \times \calC & \to \calC \\*
  (X,Y) & \mapsto X^* \otimes Y,
\end{align*}
see \cite[Sections~IX.4-IX.6]{M71} for the terminology. In particular, such a coend is given by an object $\coend \in \calC$ together with a dinatural family of structure morphisms of the form $i_X : X^* \otimes X \to \coend$, one for every $X \in \calC$. The coend $\coend$ is unique up to isomorphism, and it supports the structure of a braided Hopf algebra in $\calC$ \cite{M91, M93, L95}, meaning it admits:
\begin{enumerate}
 \item a product $\mu : \coend \otimes \coend \to \coend$ and a unit $\eta : \one \to \coend$;
 \item a coproduct $\Delta : \coend \to \coend \otimes \coend$ and a counit $\epsilon : \coend \to \one$;
 \item an antipode $S : \coend \to \coend$.
\end{enumerate}
Furthermore, it admits a non-zero integral, which is a morphism $\intL : \one \to \coend$ of $\calC$ satisfying
\[
 \mu \circ (\intL \otimes \id_\coend) = \intL \circ \epsilon =
 \mu \circ (\id_\coend \otimes \intL).
\]
As explained in \cite[Section~4.2.3]{KL01}, integrals are unique up to a scalar.

Let us denote by $\Proj(\calC)$ the full subcategory of projective objects of $\calC$, which forms an ideal in $\calC$. In other words, $\Proj(\calC)$ is absorbent under tensor products and closed under retracts. Indeed, $\Proj(\calC)$ is always the smallest non-zero ideal of $\calC$ \cite[Section~4.4]{GKP10}. Being finite, $\calC$ has enough projectives, which means every object $X \in \calC$ admits a projective cover $P_X \in \Proj(\calC)$. Furthermore, since $\calC$ is modular, it is also unimodular, which means the projective cover $P_{\one}$ of the tensor unit $\one$ is self-dual. In particular, $P_{\one}$ is in addition the injective envelope of $\one$, which means it comes equipped with both a projection morphism $\epsilon_{\one} : P_{\one} \to \one$ and an injection morphism $\eta_{\one} : \one \to P_{\one}$, see \cite[Section~2]{DGGPR19} for more details. Both $\epsilon_{\one}$ and $\eta_{\one}$ are unique up to scalars, and $\epsilon_{\one} \circ \eta_{\one} \neq 0$ if and only if $\calC$ is semisimple. 

A \textit{modified trace} on $\Proj(\calC)$ is by definition a family of linear maps
\[
 \rmt := \{ \rmt_X : \End_\calC(X) \to \Bbbk \mid X \in \Proj(\calC) \}
\]
satisfying two conditions:
\begin{enumerate}
 \item \textit{Cyclicity}: $\rmt_X(g \circ f) = \rmt_Y(f \circ g)$ for all objects $X,Y \in \Proj(\calC)$, and all morphisms $f \in \calC(X,Y)$ and $g \in \calC(Y,X)$;
 \item \textit{Partial trace}: $\rmt_{X \otimes Y}(f) = \rmt_X(\ptr(f))$ for all objects $X \in \Proj(\calC)$ and $Y \in \calC$, and every morphism $f \in \End_\calC(X \otimes Y)$, where the morphism $\ptr(f) \in \End_\calC(X)$ is defined as
  \[
   \ptr(f) := (\id_X \otimes \rev_Y) \circ (f \otimes \id_{Y^*}) \circ (\id_X \otimes \lcoev_Y).
  \]
\end{enumerate}
If $\calC$ is modular, it is in particular unimodular, and there exists a unique non-zero modified trace $\rmt$ on $\Proj(\calC)$ up to scalar, as follows from \cite[Corollary~3.2.1]{GKP11}. Furthermore, thanks to \cite[Proposition~4.2]{GR17}, $\rmt$ induces a non-degenerate pairing 
\begin{align*}
 t_X (\_ \circ \_) : \calC(Y,X) \times \calC(X,Y) & \to \Bbbk \\*
 (f,g) & \mapsto t_X(f \circ g)
\end{align*}
for all $X \in \Proj(\calC)$ and $Y \in \calC$, as opposed to the standard categorical trace, which vanishes on $\Proj(\calC)$ whenever $\calC$ is non-semisimple. We lock the normalization of the modified trace $\rmt$ and the one of the projection and the injection morphisms $\epsilon_{\one} : P_{\one} \to \one$ and $\eta_{\one} : \one \to P_{\one}$ together by setting
\[
 \rmt_{P_{\one}}(\eta_{\one} \circ \epsilon_{\one}) = 1.
\]
This is made possible by the non-degeneracy of $\rmt$, together with the fact that both $\calC(P_{\one},\one)$ and $\calC(\one,P_{\one})$ are 1-dimensional.

\subsection{3-Manifold invariants}\label{S:3-manifold_invariant}

Let us suppose from now on that $\calC$ is a modular category, and let us recall the definition of the renormalized Lyubashenko invariant $\rmL'_\calC$ constructed in \cite[Section~3.2]{DGGPR19}. In order to do so, we first need to recall how the integral $\intL$ of the coend $\coend$ can be used to construct a functor $F_\intL$ defined over the category $\calR_\intL$ of bichrome graphs, taking values in $\calC$, and extending the Reshetikhin--Turaev functor $F_\calC$ of \cite{T94}.

Bichrome graphs are a only mild generalization of standard ribbon graphs. For the purposes of the present construction, we can even choose to work with a simplified version of the notion, and just define a \textit{bichrome graph} to be the union of a \textit{red} oriented framed link, carrying no labels, and a \textit{blue} oriented ribbon graph, labeled as usual by objects of $\calC$. Red and blue components are allowed to entwine, but not to intersect. This definition is less general than the original one of \cite[Section~3.1]{DGGPR19}, which also made room for bichrome coupons. It is important to specify that all constructions appearing in this paper work perfectly fine with the general definition of bichrome graphs, and that this is in fact the correct setting for recovering the TQFTs of \cite{DGGPR19}. However, all the arguments for the construction and the characterization of ETQFTs provided here do not require bichrome coupons, so in first approximation we can simply forget them. Whichever definition we pick, isotopy classes of bichrome graphs can be organized as morphisms of a category $\calR_\intL$, and the category $\calR_\calC$ of standard ribbon graphs can be interpreted as the subcategory of $\calR_\intL$ whose morphisms are exclusively blue graphs.

The \textit{Lyubashenko--Reshetikhin--Turaev functor} 
\[
 F_\intL : \calR_\intL \to \calC
\]
is defined in \cite[Section~3.1]{DGGPR19} as an extension of the Reshetikhin--Turaev functor $F_\calC : \calR_\calC \to \calC$, in the sense that it coincides with $F_\calC$ when restricted to $\calR_\calC$. The image of a general bichrome graph $T$ featuring red components is constructed from a so-called \textit{bottom graph presentation} $\tilde{T}$ of $T$ by exploiting the universal property satisfied by the coend $\calL$ in order to extract a morphism $f_\calC(\tilde{T})$ in $\calC$, which then needs to be precomposed with an appropriate tensor power of the integral $\intL$, see \cite[Section~2.2.1]{DGGPR20} for a more detailed explanation.

We say a bichrome graph is \textit{closed} if its source and target are empty, and we say it is \textit{admissible} if it features at least a projective object of $\calC$ among the labels of its blue edges. An admissible closed bichrome graph $T$ always admits a \textit{cutting presentation}, which is another bichrome graph $T_V$ having a unique incoming boundary vertex and a unique outgoing one, both positive and labeled by a projective object $V \in \Proj(\calC)$, with trace closure equal to $T$. Cutting presentations are by no means unique. However, if $T$ is a closed admissible bichrome graph and $T_V$ is a cutting presentation, then the scalar
\[
 F'_\intL(T) := \rmt_V(F_\intL(T_V))
\]
is a topological invariant of $T$ \cite[Theorem~3.3]{DGGPR19}.

The admissible bichrome graph invariant $F'_\intL$ can be used to define an invariant $\rmL'_\calC$ of admissible decorated 3-manifolds. An \textit{admissible decorated 3-manifold} is by definition a pair $(M,T)$ with $M$ a connected closed 3-manifold and $T \subset M$ an admissible closed bichrome graph. To define the invariant, we need to use \textit{stabilization coefficients}, which are defined as
\[
 \Delta_\pm := F_\intL(O_\pm) \in \Bbbk
\]
for a red $\pm 1$-framed unknot $O_\pm$. Since these coefficients are always non-zero, as a consequence of the modularity of $\calC$ \cite[Proposition~2.6]{DGGPR19}, we can fix once and for all constants $\calD$ and $\delta$ satisfying
\begin{align*}
 \calD^2 &= \Delta_+ \Delta_-, &
 \delta &= \frac{\calD}{\Delta_-} = \frac{\Delta_+}{\calD}.
\end{align*}
If $(M,T)$ is an admissible decorated 3-manifold, and $L$ is a surgery presentation of $M$, given by a framed oriented link (assumed to be red) in $S^3$ with $\ell$ components and signature $\sigma$, then the scalar
\begin{equation*}
 \rmL'_\calC(M,T) := \calD^{-1-\ell} \delta^{-\sigma} F'_\intL(L \cup T)
\end{equation*}
is a topological invariant of the pair $(M,T)$ \cite[Theorem~3.8]{DGGPR19}. We refer to $\rmL'_\calC(M,T)$ as the \textit{renormalized Lyubashenko invariant} of the admissible decorated 3-manifold $(M,T)$.

\subsection{TQFTs}\label{S:TQFT}

We are now ready to recall the definition of the TQFT $\rmV_\calC$. Its source is provided by the \textit{category $\rmadCob_\calC$ of admissible cobordisms}. An object $\bbSigma$ of $\rmadCob_\calC$ is defined to be a triple $(\Sigma,P,\lambda)$ where: 
\begin{enumerate}
 \item $\Sigma$ is a closed surface;
 \item $P \subset \Sigma$ is a finite set of oriented framed points labeled by objects of $\calC$;
 \item $\lambda \subset H_1(\Sigma;\R)$ is a Lagrangian subspace for the intersection form.
\end{enumerate}
A morphism $\bbM : \bbSigma \rightarrow \bbSigma'$ of $\rmadCob_\calC$ is defined to be an equivalence class of admissible triples $(M,T,n)$ where:
\begin{enumerate}
 \item $M$ is a 3-dimensional cobordism from $\Sigma$ to $\Sigma'$;
 \item $T \subset M$ is a bichrome graph from $P$ to $P'$;
 \item $n \in \Z$ is an integer called the \textit{signature defect}.
\end{enumerate}
A triple $(M,T,n)$ is said to be \textit{admissible} if every connected component of $M$ disjoint from the incoming boundary $\partial_- M \cong \Sigma$ contains an admissible subgraph of $T$. Two triples $(M,T,n)$ and $(M',T',n')$ are by definition equivalent if $n = n'$ and if there exists an isomorphism of cobordisms $f : M \rightarrow M'$ satisfying $f(T) = T'$. As a consequence of the admissibility condition, surfaces without labels in $\Proj(\calC)$ are not dualizable in $\rmadCob_\calC$.

We can use the universal construction of \cite{BHMV95} to extend any function defined on $\rmadCob_\calC(\varnothing,\varnothing)$ to a functor defined on $\rmadCob_\calC$. In order to do this, we first need to extend $\rmL'_\calC$ to an invariant of closed morphisms of $\rmadCob_\calC$ by setting
\begin{equation*}
 \rmL'_\calC(\bbM) := \delta^n \rmL'_\calC(M,T)
\end{equation*}
for every closed connected morphism $\bbM = (M,T,n)$, and by setting
\[
 \rmL'_\calC(\bbM_1 \disjun \ldots \disjun \bbM_k) := \prod_{i=1}^k \rmL'_\calC(\bbM_i)
\]
for every disjoint union of closed connected morphisms $\bbM_1, \ldots, \bbM_k$. 

Then, if $\bbSigma$ is an object of $\rmadCob_\calC$, let $\calV(\bbSigma)$ denote the vector space generated by the set $\rmadCob_\calC(\varnothing,\bbSigma)$ of morphisms $\bbM_{\bbSigma} : \varnothing \rightarrow \bbSigma$ of $\rmadCob_\calC$, and let $\calV'(\bbSigma)$ denote the vector space generated by the set $\rmadCob_\calC(\bbSigma,\varnothing)$ of morphisms $\bbM'_{\bbSigma} : \bbSigma \rightarrow \varnothing$ of $\rmadCob_\calC$. Next, we consider the bilinear map
\begin{align*}
 \langle \_,\_ \rangle_{\bbSigma} : \calV'(\bbSigma) \times \calV(\bbSigma) & \to \Bbbk \\*
 (\bbM'_{\bbSigma},\bbM_{\bbSigma}) & \mapsto \rmL'_\calC(\bbM'_{\bbSigma} \circ \bbM_{\bbSigma}).
\end{align*}
State spaces are then defined as quotients with respect to the annihilators
\begin{align*}
 \rmV_\calC(\bbSigma) &:= \calV(\bbSigma) / \calV'(\bbSigma)^\perp, 
 &
 \rmV'_\calC(\bbSigma) &:= \calV'(\bbSigma) / \calV(\bbSigma)^\perp.
\end{align*}
Now, if $\bbM : \bbSigma \to \bbSigma'$ is a morphism of $\rmadCob_\calC$, its associated operators are simply given by
\begin{align*}
 \rmV_\calC(\bbM) : \rmV_\calC(\bbSigma) & \to \rmV_\calC(\bbSigma') 
 &
 \rmV'_\calC(\bbM) : \rmV'_\calC(\bbSigma') & \to \rmV'_\calC(\bbSigma) \\*
 {}[\bbM_{\bbSigma}] & \mapsto [\bbM \circ \bbM_{\bbSigma}], &
 {}[\bbM'_{\bbSigma}] & \mapsto [\bbM'_{\bbSigma} \circ \bbM].
\end{align*}
This defines a pair of functors
\begin{align*}
 \rmV_\calC &: \rmadCob_\calC \to \Vect_\Bbbk,
 & 
 \rmV'_\calC &: (\rmadCob_\calC)^\op \to \Vect_\Bbbk.
\end{align*}
As shown in \cite[Theorem~4.12]{DGGPR19}, $\rmV_\calC$ and $\rmV'_\calC$ are symmetric monoidal, which means they provide TQFTs defined on $\rmadCob_\calC$ and $(\rmadCob_\calC)^\op$, respectively.

\section{2-Category of admissible decorated cobordisms}\label{S:admissible_cobordisms}

In this section, which follows closely \cite[Chapter~2]{D17}, we introduce the symmetric monoidal $2$-cat\-e\-go\-ry $\bfadCob_\calC$ of admissible decorated cobordisms, which will be used as the source $2$-cat\-e\-go\-ry for our ETQFTs. It is convenient to start by introducing the larger symmetric monoidal $2$-cat\-e\-go\-ry $\bfCob_\calC$ of \textit{decorated cobordisms}, which contains $\bfadCob_\calC$ as a sub-$2$-cat\-e\-go\-ry. We stress the fact that $\bfCob_\calC$ is a rigid $2$-cat\-e\-go\-ry, while $\bfadCob_\calC$ is not. See \cite[Appendix~D]{D17} for the relevant definitions and notations concerning symmetric monoidal $2$-cat\-e\-go\-ries.

An \textit{object $\bbGamma$ of $\bfCob_\calC$} is defined to be a closed $1$-dimensional manifold $\Gamma$.

A \textit{$1$-mor\-phism $\bbSigma : \bbGamma \rightarrow \bbGamma'$ of $\bfCob_\calC$} between objects $\bbGamma = \Gamma$ and $\bbGamma' = \Gamma'$ is defined to be a triple $(\Sigma,P,\lambda)$ where: 
\begin{enumerate}
 \item $\Sigma$ is a $2$-dimensional cobordism from $\Gamma$ to $\Gamma'$;
 \item $P \subset \Sigma$ is a finite set of oriented framed points labeled by objects of $\calC$;
 \item $\lambda \subset H_1(\Sigma;\R)$ is a Lagrangian subspace.
\end{enumerate}
We will refer to the set $P \subset \Sigma$ as the \textit{set of marked points} of $\bbSigma$. We recall that, by definition, a \textit{Lagrangian subspace} of $H_1(\Sigma;\R)$ is a maximal isotropic subspace with respect to the transverse intersection form, which is non-degenerate if and only if $\Gamma = \Gamma' = \varnothing$.

A \textit{$2$-mor\-phism $\bbM : \bbSigma \Rightarrow \bbSigma'$ of $\bfCob_\calC$} between $1$-mor\-phisms $\bbSigma,\bbSigma' : \bbGamma \rightarrow \bbGamma'$ is defined to be an equivalence class of triples $(M,T,n)$ where:
\begin{enumerate}
 \item $M$ is a 3-dimensional cobordism with corners from $\Sigma$ to $\Sigma'$;
 \item $T \subset M$ is a bichrome graph from $P$ to $P'$;
 \item $n \in \Z$ is an integer, the signature defect.
\end{enumerate}
Two triples $(M,T,n)$ and $(M',T',n')$ are by definition equivalent if $n = n'$ and if there exists an isomorphism of cobordisms with corners $f : M \rightarrow M'$ satisfying $f(T) = T'$. For the notion of cobordism with corners used here, see \cite[Definition~B.5]{D17}.

The \textit{identity $1$-morphism $\id_{\bbGamma} : \bbGamma \rightarrow \bbGamma$ of an object $\bbGamma$ of $\bfCob_\calC$} is
\[
 \left( I \times \Gamma,\varnothing,H_1(I \times \Gamma;\R) \right),
\]
where $I$ denotes the interval $[0,1]$.

The \textit{identity $2$-morphism $\id_{\bbSigma} : \bbSigma \Rightarrow \bbSigma$ of a $1$-morphism $\bbSigma : \bbGamma \rightarrow \bbGamma'$ of $\bfCob_\calC$} is the equivalence class of
\[
 \left( \Sigma \times I,P \times I,0 \right).
\]

The \textit{horizontal composition $\bbSigma' \circ \bbSigma : \bbGamma \rightarrow \bbGamma''$ of $1$-mor\-phisms $\bbSigma : \bbGamma \rightarrow \bbGamma'$ and $\bbSigma' : \bbGamma' \rightarrow \bbGamma''$ of $\bfCob_\calC$} is
\[
 \left( \Sigma \cup_{\Gamma'} \Sigma',P \cup P',(i_\Sigma)_*(\lambda) + (i_{\Sigma'})_*(\lambda') \right),
\]
where $i_\Sigma : \Sigma \hookrightarrow \Sigma \cup_{\Gamma'} \Sigma'$ and $i_{\Sigma'} : \Sigma' \hookrightarrow \Sigma \cup_{\Gamma'} \Sigma'$ denote inclusions.

The \textit{horizontal composition $\bbM' \circ \bbM : \bbSigma' \circ \bbSigma \Rightarrow \bbSigma''' \circ \bbSigma''$ of $2$-mor\-phisms $\bbM : \bbSigma \Rightarrow \bbSigma''$ and $\bbM' : \bbSigma' \Rightarrow \bbSigma'''$ of $\bfCob_\calC$} between $1$-mor\-phisms $\bbSigma,\bbSigma'' : \bbGamma \rightarrow \bbGamma'$ and $\bbSigma',\bbSigma''' : \bbGamma' \rightarrow \bbGamma''$ is the equivalence class of
\[
 \left( M \cup_{\Gamma' \times I} M',T \cup T', n + n' \right).
\]
For the horizontal gluing of cobordisms with corners, see \cite[Definition~B.7]{D17}.

The \textit{vertical composition $\bbM' \ast \bbM : \bbSigma \Rightarrow \bbSigma''$ of $2$-mor\-phisms $\bbM : \bbSigma \Rightarrow \bbSigma'$ and $\bbM' : \bbSigma' \Rightarrow \bbSigma''$ of $\bfCob_\calC$} between $1$-mor\-phisms $\bbSigma,\bbSigma',\bbSigma'' : \bbGamma \rightarrow \bbGamma'$ is the equivalence class of
\[
 \left( M \cup_{\Sigma'} M',T \cup_{P'} T', n + n' - \mu(M_*(\lambda),\lambda',(M')^*(\lambda'')) \right),
\]
where the push-forward $M_*(\lambda)$ and the pull-back $(M')^*(\lambda'')$ are defined by
\begin{align*}
 M_*(\lambda) &:= \{ x' \in H_1(\Sigma';\R) \mid (j_{\Sigma'})_*(x') \in (j_\Sigma)_* (\lambda) \}, \\*
 (M')^*(\lambda'') &:= \{ x' \in H_1(\Sigma';\R) \mid (j'_{\Sigma'})_*(x') \in (j'_{\Sigma''})_*(\lambda'') \}
\end{align*}
for the embeddings
\[
 j_{\Sigma} : \Sigma \hookrightarrow M, \quad
 j_{\Sigma'} : \Sigma' \hookrightarrow M, \quad
 j'_{\Sigma'} : \Sigma' \hookrightarrow M', \quad
 j'_{\Sigma''} : \Sigma'' \hookrightarrow M'\phantom{,}
\]
induced by horizontal boundary identifications, and where $\mu$ is the \textit{Maslov index} of the three Lagrangian subspaces, whose definition extends word-by-word to the degenerate case $\Gamma \neq \varnothing \neq \Gamma'$, see \cite[Section~C.3]{D17}. For the vertical gluing of cobordisms with corners, see \cite[Definition~B.8]{D17}.

The \textit{tensor unit of $\bfCob_\calC$} is the empty $1$-dimensional manifold $\varnothing$.

The \textit{tensor product $\bbGamma \disjun \bbGamma'$ of objects $\bbGamma$ and $\bbGamma'$ of $\bfCob_\calC$} is the $1$-dimensional manifold $\Gamma \sqcup \Gamma'$.

The \textit{tensor product $\bbSigma \disjun \bbSigma' : \bbGamma \disjun \bbGamma' \rightarrow \bbGamma'' \disjun \bbGamma'''$ of $1$-mor\-phisms $\bbSigma : \bbGamma \rightarrow \bbGamma''$ and $\bbSigma' : \bbGamma' \rightarrow \bbGamma'''$ of $\bfCob_\calC$} is 
\[
 (\Sigma \sqcup \Sigma',P \sqcup P',\lambda \oplus \lambda').
\]

The \textit{tensor product $\bbM \disjun \bbM' : \bbSigma \disjun \bbSigma' \Rightarrow \bbSigma'' \disjun \bbSigma'''$ of $2$-morphisms $\bbM : \bbSigma \Rightarrow \bbSigma''$ and $\bbM' : \bbSigma' \Rightarrow \bbSigma'''$ of $\bfCob_\calC$} is the equivalence class of
\[
 (M \sqcup M',T \sqcup T',n+n').
\]

The symmetric braiding of $\bfCob_\calC$ is completely standard, and it plays no non-trivial role in our construction.

Next, let us define the symmetric monoidal $2$-cat\-e\-go\-ry $\bfadCob_\calC$ of admissible decorated cobordisms. This requires fixing our terminology first. If $\Sigma$ is a $2$-di\-men\-sion\-al cobordism, then we say a set $P \subset \Sigma$ of marked points is \textit{admissible} if it features a projective object of $\calC$ among its labels. A $1$-mor\-phism $\bbSigma = (\Sigma,P,\lambda)$ of $\bfCob_\calC$ is said to be \textit{admissible} if every connected component of $\Sigma$ disjoint from the incoming boundary contains an admissible subset of $P$, and it is said to be \textit{strongly admissible} if the same condition holds more generally for all connected components of $\Sigma$, regardless of their intersection with the incoming boundary. Similarly, a $2$-mor\-phism $\bbM = (M,T,n)$ of $\bfCob_\calC$ is said to be \textit{admissible} if every connected component of $M$ disjoint from the incoming horizontal boundary contains an admissible bichrome subgraph of $T$, and it is said to be \textit{strongly admissible} if the same condition holds more generally for all connected components of $M$, regardless of their intersection with the incoming horizontal boundary.

\begin{definition}\label{D:sym_mon_2-cat_of_adm_cob}
 The \textit{symmetric monoidal $2$-cat\-e\-go\-ry of admissible cobordisms} is the symmetric mo\-n\-oi\-dal sub-$2$-cat\-e\-go\-ry $\bfadCob_\calC$ of $\bfCob_\calC$ whose set of objects coincides with the one of $\bfCob_\calC$, whose categories of morphisms are provided by admissible morphisms of $\bfCob_\calC$, and whose symmetric mo\-n\-oi\-dal structure is directly inherited from the one of $\bfCob_\calC$.
\end{definition}

\section{2-Category of complete linear categories}\label{S:complete_lin_cat}

In this section, which is extracted from \cite[Section~E.2]{D17}, we recall the definition of the symmetric monoidal $2$-cat\-e\-go\-ry $\smash{\coCat_\Bbbk}$ of complete linear categories, which will provide the target $2$-cat\-e\-go\-ry for our ETQFTs. It is useful to start by recalling the definition of the larger symmetric monoidal $2$-cat\-e\-go\-ry $\bfCat_\Bbbk$ of linear categories, which contains $\smash{\coCat_\Bbbk}$ as a sub-$2$-cat\-e\-go\-ry, although not as a monoidal one. This is due to the fact that the restriction of the tensor product of $\bfCat_\Bbbk$ to the sub-$2$-cat\-e\-go\-ry $\coCat_\Bbbk$ does not have $\smash{\coCat_\Bbbk}$ as target. Again, a reference for all the relevant definitions and notations concerning symmetric monoidal $2$-cat\-e\-go\-ries is provided by \cite[Appendix~D]{D17}.

An object $A$ of $\bfCat_\Bbbk$ is a \textit{linear category}, that is, a $\Vect_\Bbbk$-enriched category. A $1$-mor\-phism $F : A \rightarrow A'$ of $\bfCat_\Bbbk$ between linear categories $A$ and $A'$ is a \textit{linear functor}, that is, a $\Vect_\Bbbk$-enriched functor. A $2$-mor\-phism $\eta : F \Rightarrow F'$ of $\bfCat_\Bbbk$ between linear functors $F,F' : A \rightarrow A'$ is simply an ordinary natural transformation. The horizontal composition of $\bfCat_\Bbbk$ is denoted $\circ$, and it is given by the standard composition of functors and by the horizontal composition of natural transformations. Similarly, the vertical composition of $\bfCat_\Bbbk$ is denoted $\ast$, and it is given by the standard vertical composition of natural transformations. The tensor unit of $\bfCat_\Bbbk$ is denoted $\Bbbk$, and it is given by any linear category having a single simple object, that is, an object whose endomorphism space is isomorphic to the base field $\Bbbk$. The tensor product of $\bfCat_\Bbbk$ is denoted $\sqtimes$, and it is given by the standard $\Vect_\Bbbk$-en\-riched tensor product of linear categories. The symmetric braiding of $\bfCat_{\Bbbk}$ is again completely standard, and will not be discussed.

We say a linear category is \textit{Cauchy complete}, or simply \textit{complete} for short, if it is closed under finite direct sums of objects, and if all its idempotent endomorphisms are split. The \textit{$2$-cat\-e\-go\-ry of complete linear categories} is then defined to be the full sub-$2$-cat\-e\-go\-ry $\coCat_\Bbbk$ of $\bfCat_{\Bbbk}$ whose objects are complete linear categories.

If a linear category is not complete, then there is a canonical \textit{completion} procedure which consists in manually adding direct sums of objects, as well as subobjects corresponding to images of idempotent endomorphisms. The construction is carried out in two steps.

First of all, the \textit{additive completion $\Mat(A)$ of a linear category $A$} is the linear category whose objects are vectors
\[
 \bigoplus_{i \in X} x_i := \left( \begin{array}{c}
                                    \vdots \\ x_i \\ \vdots
                                   \end{array} \right)
\]
with $x_i \in A$ for every $i$ in some finite ordered set $X$, and whose morphisms from $\smash{\displaystyle \bigoplus_{i \in X} x_i}$ to $\smash{\displaystyle \bigoplus_{i \in Y} y_i}$ are matrices
\[
 \left( f_{ij} \right)_{(i,j) \in Y \times X} := 
 \left( \begin{array}{ccc}
         \ddots & \vdots & \bdots \\
         \cdots & f_{ij} & \cdots \\
         \bdots & \vdots & \ddots
        \end{array} \right)
\]
with $f_{ij} \in A(x_j,y_i)$ for every $(i,j) \in Y \times X$. The identity of a vector-object $\smash{\displaystyle \bigoplus_{i \in X} x_i \in \Mat(A)}$ is the matrix-morphism
\[
 \left( \delta_{ij} \cdot \id_{x_i} \right)_{(i,j) \in X^2} \in \End_{\Mat(A)} \left( \bigoplus_{i \in X} x_i \right),
\] 
and the composition of a matrix-morphism
\[
 \left( f_{ij} \right)_{(i,j) \in Y \times X} 
 \in \Mat(A) \left( \bigoplus_{i \in X} x_i, \bigoplus_{i \in Y} y_i \right)
\]
with a matrix-morphism
\[
 \left( g_{ij} \right)_{(i,j) \in Z \times Y} 
 \in \Mat(A) \left( \bigoplus_{i \in Y} y_i, \bigoplus_{i \in Z} z_i \right)
\]
is the matrix-morphism
\[
 \left( \sum_{k \in Y} g_{ik} \circ f_{kj} \right)_{(i,j) \in Z \times X} 
 \in \Mat(A) \left( \bigoplus_{i \in X} x_i, \bigoplus_{i \in Z} z_i \right).
\]
Every object of $A$ can be canonically interpreted as a vector-object of $\Mat(A)$ with a single entry.

Next, the \textit{idempotent completion $\Kar(A)$ of a linear category $A$}, also known as the \textit{Karoubi envelope}, is the linear category whose objects are pairs $\im(p) := (x,p)$ with $x \in A$ and $p \in \End_A(x)$ satisfying $p \circ p = p$, and whose morphisms from $\im(p)$ to $\im(q)$ are $f \in A(x,y)$ satisfying $f \circ p = f = q \circ f$. The identity of an image-object $\im(p) \in \Kar(A)$ is the idempotent endomorphism $p \in \End_{\Kar(A)}(\im(p))$ itself, and composition is directly inherited from the composition in $A$. Again, every object of $A$ can be canonically interpreted as the image-object of $\Kar(A)$ determined by its identity endomorphism.

The \textit{completion $\hat{A}$ of a linear category $A$} is then defined to be the complete linear category $\Kar(\Mat(A))$.

Clearly, every linear functor $F : A \rightarrow A'$ extends canonically to a linear functor $\hat{F} : \hat{A} \rightarrow \hat{A}'$, called the \textit{completion of $F$}, and similarly, every natural transformation $\eta : F \Rightarrow F'$ extends canonically to a natural transformation $\hat{\eta} : \hat{F} \Rightarrow \hat{F}'$, called the \textit{completion of $\eta$}, see \cite[Section~E.2]{D17} for both definitions.

\begin{proposition}\label{P:co_strict_2-funct}
 The operation of completion of linear categories, linear functors, and natural transformations defines a strict $2$-func\-tor $\cmpl : \bfCat_{\Bbbk} \rightarrow \coCat_\Bbbk$.
\end{proposition}

See \cite[Proposition~E.1]{D17} for a proof. This result allows us to extend the operation of completion to $2$-func\-tors with target $\bfCat_\Bbbk$, as well as to $2$-trans\-for\-ma\-tions between them. Indeed, for every $2$-category $\bcalC$, every $2$-func\-tor $\bfF : \bcalC \rightarrow \bfCat_\Bbbk$ automatically induces a $2$-func\-tor $\hat{\bfF} : \bcalC \rightarrow \coCat_\Bbbk$, which we call the \textit{completion} of $\bfF$, and which is simply defined as the composition $\cmpl \circ \bfF$. Similarly, every $2$-trans\-for\-ma\-tion $\bfsigma : \bfF \Rightarrow \bfF'$ between $2$-func\-tors $\bfF,\bfF' : \bcalC \rightarrow \bfCat_\Bbbk$ induces a $2$-trans\-for\-ma\-tion $\hat{\bfsigma} : \hat{\bfF} \Rightarrow \hat{\bfF}'$, which we call the \textit{completion} of $\bfsigma$, and which is simply defined as the left whiskering $\cmpl \triangleright \bfsigma$, see \cite[Definition~D.6]{D17}.

\begin{definition}\label{D:sym_mon_2-cat_of_co_lin_cat}
 The \textit{symmetric monoidal $2$-cat\-e\-go\-ry of complete linear categories} is obtained from the full sub-$2$-cat\-e\-go\-ry $\coCat_\Bbbk$ of $\bfCat_{\Bbbk}$ by specifying the completion $\hat{\Bbbk}$ of the tensor unit $\Bbbk$ of $\bfCat_\Bbbk$ as a tensor unit, and by specifying the completion $\csqtimes$ of the tensor product $\sqtimes$ of $\bfCat_\Bbbk$ as a tensor product.
\end{definition}

\begin{proposition}\label{P:completion}
 The strict $2$-functor $\cmpl : \bfCat_{\Bbbk} \rightarrow \coCat_{\Bbbk}$ is symmetric mo\-n\-oi\-dal.
\end{proposition}

See \cite[Section~E.5]{D17} for a proof. The completion $2$-func\-tor $\cmpl$ allows us to focus on objects of $\bfCat_{\Bbbk}$, provided we consider them to be equivalent whenever their completions are equivalent in $\smash{\coCat_{\Bbbk}}$. This motivates the next definition.

We say two linear categories $A$ and $A'$ are \textit{c-equivalent} if their completions $\smash{\hat{A}}$ and $\smash{\hat{A}'}$ are equivalent, and we say a linear functor $F : A \rightarrow A'$ is a \textit{c-equivalence} if its completion $\smash{\hat{F} : \hat{A} \rightarrow \hat{A}'}$ is an equivalence. Then, let us give a very useful c-equivalence criterion. We recall that a set $D$ of objects of a linear category $A$ is said to be a \textit{dominating set} if for every $x \in A$ there exist objects $x_1, \ldots, x_m \in D$ and morphisms $f_i \in A(x,x_i)$ and $f'_i \in A(x_i,x)$ satisfying
\[
 \id_x = \sum_{i=1}^m f'_i \circ f_i,
\]
in which case we also say \textit{$D$ dominates $A$}.

\begin{proposition}\label{P:Morita_equivalence}
 Let $F : A \rightarrow A'$ be a fully faithful linear functor. If the set of objects $F(A)$ dominates $A'$, then $F$ is a c-equivalence.
\end{proposition}

See \cite[Proposition~E.3]{D17} for a proof.

\section{Extended universal construction}\label{S:extended_universal_construction}

In this section, which is adapted from \cite[Section~3.4]{D17}, we recall how the TQFT $\rmV_\calC : \rmadCob_\calC \rightarrow \Vect_\Bbbk$ can be extend to a pair of 2-func\-tors 
\begin{align*}
 \bfA_\calC &: \bfadCob_\calC \rightarrow \bfCat_\Bbbk,
 &
 \bfA'_\calC &: (\bfadCob_\calC)^\op \rightarrow \bfCat_\Bbbk
\end{align*}
following a procedure called the \textit{extended universal construction}, which is inspired by the universal construction of \cite{BHMV95}. In the following, we quickly recall the layer-by-layer construction of both $\bfA_\calC$ and $\bfA'_\calC$.

Let us start by considering an object $\bbGamma$ of $\bfadCob_\calC$. Let $\calA(\bbGamma)$ denote the linear category generated by the category $\bfadCob_\calC(\varnothing,\bbGamma)$ of $1$-mor\-phisms $\bbSigma_{\bbGamma} : \varnothing \rightarrow \bbGamma$ and $2$-mor\-phisms $\bbM_{\bbGamma} : \bbSigma_{\bbGamma} \Rightarrow \bbSigma''_{\bbGamma}$ of $\bfadCob_\calC$, and let $\calA'(\bbGamma)$ denote the linear category generated by the category $\bfadCob_\calC(\bbGamma,\varnothing)$ of $1$-mor\-phisms $\bbSigma'_{\bbGamma} : \bbGamma \rightarrow \varnothing$ and $2$-mor\-phisms $\bbM'_{\bbGamma} : \bbSigma'_{\bbGamma} \Rightarrow \bbSigma'''_{\bbGamma}$ of $\bfadCob_\calC$. Here, recall that the linear category generated by another category is simply defined as the linear category with the same set of objects, and with vector spaces of morphisms generated by the corresponding sets of morphisms. Next, we consider the bilinear functor
\[
 \langle \_ , \_ \rangle_{\bbGamma} : \calA'(\bbGamma) \times \calA(\bbGamma) \to \Vect_\Bbbk,
\]
sending every object $(\bbSigma'_{\bbGamma}, \bbSigma_{\bbGamma}) \in \calA'(\bbGamma) \times \calA(\bbGamma)$ to the state space
\[
 \rmV_\calC(\bbSigma'_{\bbGamma} \circ \bbSigma_{\bbGamma}) \in \Vect_\Bbbk,
\]
and sending every morphism $(\bbM'_{\bbGamma},\bbM_{\bbGamma}) \in \calA'(\bbGamma)(\bbSigma'_{\bbGamma},\bbSigma'''_{\bbGamma}) \times \calA(\bbGamma)(\bbSigma_{\bbGamma},\bbSigma''_{\bbGamma})$ to the linear operator
\[
 \rmV_\calC(\bbM'_{\bbGamma} \circ \bbM_{\bbGamma}) : \rmV_\calC(\bbSigma'_{\bbGamma} \circ \bbSigma_{\bbGamma}) \to \rmV_\calC(\bbSigma'''_{\bbGamma} \circ \bbSigma''_{\bbGamma}).
\] 
The annihilator of $\calA'(\bbGamma)$ with respect to $\langle \_ , \_ \rangle_{\bbGamma}$ is the linear congruence $\calA'(\bbGamma)^{\perp}$ on $\calA(\bbGamma)$ defined by
\[
 \calA'(\bbGamma)^{\perp}(\bbSigma_{\bbGamma},\bbSigma''_{\bbGamma}) := \{ \bbM_{\bbGamma} \in \calA(\bbGamma)(\bbSigma_{\bbGamma},\bbSigma''_{\bbGamma}) \mid \langle \id_{\bbSigma'_{\bbGamma}},\bbM_{\bbGamma} \rangle = 0 \Forall \bbSigma'_{\bbGamma} \in \calA'(\bbGamma) \},
\]
and analogously the annihilator of $\calA(\bbGamma)$ with respect to $\langle \_ , \_ \rangle_{\bbGamma}$ is the linear congruence $\calA(\bbGamma)^{\perp}$ on $\calA'(\bbGamma)$ defined by
\[
 \calA(\bbGamma)^{\perp}(\bbSigma'_{\bbGamma},\bbSigma'''_{\bbGamma}) := \{ \bbM'_{\bbGamma} \in \calA'(\bbGamma)(\bbSigma'_{\bbGamma},\bbSigma'''_{\bbGamma}) \mid \langle \bbM'_{\bbGamma},\id_{\bbSigma_{\bbGamma}} \rangle = 0 \Forall \bbSigma_{\bbGamma} \in \calA(\bbGamma) \}.
\]
Here, the term congruence is used in the sense of \cite[Section~II.8]{M71}. Then, the image of $\bbGamma$ under $\bfA_\calC$ is defined to be the linear category
\[
 \bfA_\calC(\bbGamma) := \calA(\bbGamma)/\calA'(\bbGamma)^{\perp},
\]
while the image of $\bbGamma$ under $\bfA'_\calC$ is defined to be the linear category
\[
 \bfA'_\calC(\bbGamma) := \calA'(\bbGamma)/\calA(\bbGamma)^{\perp}.
\]
Here, recall that the quotient of a linear category with respect to a linear congruence is simply defined as the linear category with the same set of objects, and with vector spaces of morphisms obtained as quotients of the original ones with respect to the linear subspaces determined by the linear congruence.

Next, let us consider a 1-morphism $\bbSigma : \bbGamma \rightarrow \bbGamma'$ of $\bfadCob_\calC$. The image of $\bbSigma$ under $\bfA_\calC$ is defined to be the linear functor
\[
 \bfA_\calC(\bbSigma) : \bfA_\calC(\bbGamma) \rightarrow \bfA_\calC(\bbGamma')
\]
sending every object $\bbSigma_{\bbGamma} \in \bfA_\calC(\bbGamma)$ to the object
\[
 \bbSigma \circ \bbSigma_{\bbGamma} \in \bfA_\calC(\bbGamma'),
\]
and every morphism $\left[ \bbM_{\bbGamma} \right] \in \bfA_\calC(\bbGamma)(\bbSigma_{\bbGamma},\bbSigma''_{\bbGamma})$ to the morphism
\[
 \left[ \id_{\bbSigma} \circ \bbM_{\bbGamma} \right] \in \bfA_\calC(\bbGamma')(\bbSigma \circ \bbSigma_{\bbGamma},\bbSigma \circ \bbSigma''_{\bbGamma}),
\]
while the image of $\bbSigma$ under $\bfA'_\calC$ is defined to be the linear functor
\[
 \bfA'_\calC(\bbSigma) : \bfA'_\calC(\bbGamma') \rightarrow \bfA'_\calC(\bbGamma)
\]
sending every object $\bbSigma'_{\bbGamma'} \in \bfA'_\calC(\bbGamma')$ to the object 
\[
 \bbSigma'_{\bbGamma'} \circ \bbSigma \in \bfA'_\calC(\bbGamma),
\]
and every morphism $\left[ \bbM'_{\bbGamma'} \right] \in \bfA'_\calC(\bbGamma')(\bbSigma'_{\bbGamma'},\bbSigma'''_{\bbGamma'})$ to the morphism
\[
 \left[ \bbM'_{\bbGamma'} \circ \id_{\bbSigma} \right] \in \bfA'_\calC(\bbGamma)(\bbSigma'_{\bbGamma'} \circ \bbSigma, \bbSigma'''_{\bbGamma'} \circ \bbSigma).
\]

Finally, let us consider a 2-morphism $\bbM : \bbSigma \Rightarrow \bbSigma'$ of $\bfadCob_\calC$ between 1-morphisms $\bbSigma, \bbSigma' : \bbGamma \rightarrow \bbGamma'$. The image of $\bbM$ under $\bfA_\calC$ is defined to be the natural transformation
\[
 \bfA_\calC(\bbM) : \bfA_\calC(\bbSigma) \Rightarrow \bfA_\calC(\bbSigma')
\]
associating with every object $\bbSigma_{\bbGamma} \in \bfA_\calC(\bbSigma)$ the morphism 
\[
 \left[ \bbM \circ \id_{\bbSigma_{\bbGamma}} \right] \in \bfA_\calC(\bbGamma')(\bbSigma \circ \bbSigma_{\bbGamma},\bbSigma' \circ \bbSigma_{\bbGamma}),
\]
while the image of $\bbM$ under $\bfA'_\calC$ is defined to be the natural transformation
\[
 \bfA'_\calC(\bbM) : \bfA'_\calC(\bbSigma) \Rightarrow \bfA'_\calC(\bbSigma')
\]
associating with every object $\bbSigma'_{\bbGamma'} \in \bfA'_\calC(\bbSigma)$ the morphism
\[
 \left[ \id_{\bbSigma'_{\bbGamma'}} \circ \bbM \right] \in \bfA'_\calC(\bbGamma)(\bbSigma'_{\bbGamma'} \circ \bbSigma,\bbSigma'_{\bbGamma'} \circ \bbSigma').
\]

In Section~\ref{S:monoidality}, we will show that the $2$-func\-tor $\hat{\bfA}_\calC$ is an ETQFT, in the sense that supports the structure of a symmetric monoidal $2$-func\-tor as in \cite[Definition~D.16]{D17}. In the meantime, some of its features can be derived straight away. Indeed, let us consider the $2$-trans\-for\-ma\-tion
\[
 \bfmu : \sqtimes \circ \left( \bfA_\calC \times \bfA_\calC \right) \Rightarrow \bfA_\calC \circ \disjun
\]
defined as follows: for all objects $\bbGamma$ and $\bbGamma'$ of $\bfadCob_\calC$, we specify the linear functor
\[
 \bfmu_{\bbGamma,\bbGamma'} : \bfA_\calC(\bbGamma) \sqtimes \bfA_\calC(\bbGamma') \rightarrow \bfA_\calC(\bbGamma \disjun \bbGamma')
\]
sending every object $(\bbSigma_{\bbGamma},\bbSigma_{\bbGamma'})$ of $\bfA_\calC(\bbGamma) \sqtimes \bfA_\calC(\bbGamma')$ to the object $\bbSigma_{\bbGamma} \disjun \bbSigma_{\bbGamma'}$ of $\bfA_\calC(\bbGamma \disjun \bbGamma')$, and every morphism 
\[
 \left[ \bbM_{\bbGamma} \right] \sqtimes \left[ \bbM_{\bbGamma'} \right]
\]
of $\bfA_\calC(\bbGamma)(\bbSigma_{\bbGamma},\bbSigma''_{\bbGamma}) \sqtimes \bfA_\calC(\bbGamma')(\bbSigma_{\bbGamma'},\bbSigma''_{\bbGamma'})$ 
to the morphism
\[
 \left[ \bbM_{\bbGamma} \disjun \bbM_{\bbGamma'} \right]
\]
of $\bfA_\calC(\bbGamma \disjun \bbGamma')(\bbSigma_{\bbGamma} \disjun \bbSigma_{\bbGamma'},\bbSigma''_{\bbGamma} \disjun \bbSigma''_{\bbGamma'})$. Next, for all $1$-mor\-phisms $\bbSigma : \bbGamma \rightarrow \bbGamma''$ and $\bbSigma' : \bbGamma' \rightarrow \bbGamma'''$ of $\bfadCob_\calC$, we specify the natural transformation 
\[
 \bfmu_{\bbSigma,\bbSigma'} : \bfA_\calC(\bbSigma \disjun \bbSigma') \circ \bfmu_{\bbGamma,\bbGamma'} \Rightarrow \bfmu_{\bbGamma'',\bbGamma'''} \circ \left( \bfA_\calC(\bbSigma) \sqtimes \bfA_\calC(\bbSigma') \right)
\]
simply associating with every object $(\bbSigma_{\bbGamma},\bbSigma_{\bbGamma'})$ of $\bfA_\calC(\bbGamma) \sqtimes \bfA_\calC(\bbGamma')$ the morphism of $\bfA_\calC(\bbGamma'' \disjun \bbGamma''')((\bbSigma \disjun \bbSigma') \circ (\bbSigma_{\bbGamma} \disjun \bbSigma_{\bbGamma'}),{(\bbSigma \circ \bbSigma_{\bbGamma})} \disjun {(\bbSigma' \circ \bbSigma_{\bbGamma'})})$ defined as the equivalence class of the coherence $2$-mor\-phism provided by the monoidal structure of the $2$-cat\-e\-go\-ry $\bfadCob_\calC$.

\begin{proposition}\label{P:lax_monoidality_ETQFT}
 The linear functor $\bfmu_{\bbGamma,\bbGamma'} : \bfA_\calC(\bbGamma) \sqtimes \bfA_\calC(\bbGamma') \rightarrow \bfA_\calC(\bbGamma \disjun \bbGamma')$ is faithful for all $\bbGamma, \bbGamma' \in \bfadCob_\calC$.
\end{proposition}

\begin{proof}
 The proof is the same as the one for \cite[Proposition~3.3]{D17}. Indeed, let us consider a trivial morphism of the form
 \[
  \sum_{i=1}^m \alpha_i \cdot \left[ \bbM_{\bbGamma,i} \disjun \bbM_{\bbGamma',i} \right] \in \bfA_\calC(\bbGamma \disjun \bbGamma')(\bbSigma_{\bbGamma} \disjun \bbSigma_{\bbGamma'},\bbSigma''_{\bbGamma} \disjun \bbSigma''_{\bbGamma'}).
 \]
 This means, by definition, that the linear map 
 \[
  \sum_{i=1}^m \alpha_i \cdot \rmV_\calC \left( \id_{\bbSigma'_{\bbGamma \disjun \bbGamma'}} \circ (\bbM_{\bbGamma,i} \disjun \bbM_{\bbGamma',i}) \right)
 \]
 is zero for every object $\bbSigma'_{\bbGamma \disjun \bbGamma'} \in \bfA'_\calC(\bbGamma \disjun \bbGamma')$. In particular, this holds for every object of the form $\bbSigma'_{\bbGamma} \disjun \bbSigma'_{\bbGamma'} \in \bfA'_\calC(\bbGamma \disjun \bbGamma')$ too. Therefore, the morphism
 \[
  \sum_{i=1}^m \alpha_i \cdot \left[ \bbM_{\bbGamma,i} \right] \sqtimes \left[ \bbM_{\bbGamma',i} \right] \in \bfA_\calC(\bbGamma)(\bbSigma_{\bbGamma},\bbSigma''_{\bbGamma}) \sqtimes \bfA_\calC(\bbGamma')(\bbSigma_{\bbGamma'},\bbSigma''_{\bbGamma'})
 \]
 is also trivial.
\end{proof}

\section{Connection, domination, and triviality}\label{S:combinatorial_topological_properties}

In this section, which follows the footsteps of \cite[Section~4.3]{D17}, we discuss some properties of the $2$-func\-tor $\bfA_\calC : \bfadCob_\calC \rightarrow \bfCat_\Bbbk$ related to the behavior of the invariant $\rmL'_\calC$. For every object $\bbGamma$ of $\bfadCob_\calC$, these results can be roughly summed up as follows: 
\begin{enumerate}
 \item The vector space of morphisms between any two objects of $\bfA_\calC(\bbGamma)$ is generated by admissible bichrome graphs inside a fixed non-empty connected $3$-di\-men\-sion\-al cobordism with corners, as proved in Lemma~\ref{L:connection_lemma};
 \item The linear category $\bfA_\calC(\bbGamma)$ is dominated by admissible sets of marked points inside a fixed non-empty $2$-di\-men\-sion\-al cobordism, as proved in Lemma~\ref{L:Morita_reduction}.
 \item To check if a morphism of $\bfA_\calC(\bbGamma)$ is trivial, it is sufficient to test it against bichrome graphs inside a fixed triple of connected $3$-di\-men\-sion\-al cobordisms with corners, as proved in Lemma~\ref{L:triviality_lemma}.
\end{enumerate}
These properties will be used in Section~\ref{S:monoidality} in order to show the monoidality of $\hat{\bfA}_\calC$.

We start by introducing an equivalence relation between linear combinations of bi\-chrome graphs embedded inside $3$-di\-men\-sion\-al cobordisms with corners called skein equivalence. This notion is completely analogous to the one discussed in \cite[Section~4.5]{DGGPR19}, where more details can be found. We recall that, if $\calR_\intL$ denotes the category of bichrome graphs, we say two linear combinations of morphisms between the same pair of objects of $\calR_\intL$ are \textit{skein equivalent} if their image under the Lyubashenko--Reshetikhin--Turaev functor $F_\intL : \calR_\intL \to \calC$ coincides, see \cite[Section~4.2]{DGGPR19} for a more precise statement. Then, loosely speaking, we say two linear combinations of bichrome graphs embedded inside a fixed $3$-di\-men\-sion\-al cobordism with corners $M$ are \textit{skein equivalent} if they are related by a finite sequence of skein equivalences of $\calR_\intL$, each taking place inside a ball embedded into $M$. For a more careful explanation of this sketchy idea, see \cite[Section~4.5]{DGGPR19}. In general, if $T$ and $T'$ are linear combinations of bichrome graphs (possibly embedded inside a $3$-di\-men\-sion\-al cobordism with corners), the notation $T \doteq T'$ will stand for \textit{$T$ is skein equivalent to $T'$}.

When working with linear combinations of admissible bichrome graphs, the relevant equivalence relation will be a different one, called admissible skein equivalence. We say two linear combinations of bichrome graphs inside a fixed $3$-di\-men\-sion\-al cobordism with corners $M$ are \textit{admissibly skein equivalent} if they are related by a finite sequence of skein equivalences of $\calR_\intL$, each taking place inside a ball whose complement in $M$ contains an admissible bichrome graph. We should focus on this particular notion of skein equivalence every time we want to avoid the risk of inadvertently turning admissible bichrome graphs into non-admissible ones.

Next, we recall that \cite[Figure~6]{DGGPR19} defines an operation on admissible morphisms of $\calR_\intL$ which we refer to as the \textit{red-to-blue operation}. Roughly speaking, as suggested by the name, this operation allows us to replace red components of admissible bichrome graphs with blue ones. Since this will be crucial for our construction, it will be convenient to slightly broaden the meaning of the term \textit{skein equivalence}, both in its original form and in its admissible one. Therefore, we also declare two bichrome graphs inside a fixed $3$-di\-men\-sion\-al cobordism with corners $M$ to be \textit{skein equivalent} whenever they are related by a finite sequence of red-to-blue operations, each taking place inside a solid torus embedded into $M$. Remark that red-to-blue operations are admissible, meaning that they never turn admissible bichrome graphs into non-admissible ones.

Now, if $\bbSigma = (\Sigma,P,\lambda)$ and $\bbSigma' = (\Sigma',P',\lambda')$ are $1$-mor\-phisms of $\bfadCob_\calC$ between objects $\bbGamma$ and $\bbGamma'$, and $M$ is a $3$-di\-men\-sion\-al cobordism with corners from $\Sigma$ to $\Sigma'$, we denote with $\adSk(M;P,P')$ the \textit{admissible skein module of $M$ relative to $P$ and $P'$}, which is by definition the quotient of the vector space generated by the set of all admissible bichrome graphs in $M$ from $P$ to $P'$ modulo admissible skein equivalences.

\begin{lemma}\label{L:connection_lemma}
 If $\bbSigma_{\bbGamma} = (\Sigma_\Gamma,P,\lambda)$ and $\bbSigma''_{\bbGamma} = (\Sigma''_\Gamma,P'',\lambda'')$ are ob\-jects of $\bfA_\calC(\bbGamma)$, and if $M$ is a non-emp\-ty connected $3$-di\-men\-sion\-al cobordism with corners from $\Sigma_\Gamma$ to $\Sigma'_\Gamma$, then the linear map 
 \[
  \begin{array}{rccc}
   \pi_M : & \adSk(M;P,P'') & 
   \to & \bfA_\calC(\bbGamma)(\bbSigma_{\bbGamma},\bbSigma''_{\bbGamma}) \\
   & T & \mapsto & [M,T,0]
  \end{array}
 \]
 is surjective.
\end{lemma}

\begin{proof}
 The proof is completely analogous to that of \cite[Proposition~4.11]{DGGPR19}. First, the fact that this map is well-defined follows from the definition of the invariant $\rmL'_\calC$ in terms of the Lyubashenko--Reshetikhin--Turaev functor $F_\intL$. Next, in order to prove that $\pi_M$ is surjective, we have to show that for every mor\-phism $[M_\Gamma,T,n]$ of $\bfA_\calC(\bbGamma)(\bbSigma_{\bbGamma},\bbSigma''_{\bbGamma})$ there exist ad\-mis\-si\-ble bichrome graphs $T_1,\ldots,T_m$ in $M$ from $P$ to $P''$, and scalars $\alpha_1, \ldots, \alpha_m \in \Bbbk$ such that
 \[
  \sum_{i = 1}^m \alpha_i \cdot [M,T_i,0] = [M_{\Gamma},T,n].
 \]
 First of all, we can always suppose that $M_\Gamma$ is connected. Indeed, if it is not, then, by applying the \textit{projective trick} of \cite[Figure~8]{DGGPR19}, we can suppose that every connected component of $M_\Gamma$ contains blue coupons colored with either the projective cover morphism $\epsilon_{\one} : P_{\one} \to \one$ or the injective envelope morphism $\eta_{\one} : \one \to P_{\one}$.  Thus, thanks to \cite[Proposition~4.10]{DGGPR19}, a finite sequence of index $1$ surgeries connecting the components of $(M_\Gamma,T,n)$ determines a connected vector of $\bfA_\calC(\bbGamma)(\bbSigma_{\bbGamma},\bbSigma''_{\bbGamma})$ which is a non-zero scalar multiple of $[M_\Gamma,T,n]$. Then, since we are now assuming that $M_\Gamma$ is connected, we know that there exists a surgery presentation of $M_\Gamma$ given by a red oriented framed link $L \subset M$. This means that, thanks again to \cite[Proposition~4.10]{DGGPR19}, if $\ell$ denotes the number of components of $L$, there exists an integer $n_L \in \Z$ such that 
 \[
  [M_\Gamma,T,n] = \calD^{-\ell} \cdot [M,L \cup T,n_L] = \calD^{-\ell} \delta^{n_L} \cdot [M,L \cup T,0]
 \]
 for the non-zero coefficients $\calD, \delta \in \Bbbk$ fixed in Section~\ref{S:3-manifold_invariant}.
\end{proof}

We point out that a direct translation of Lemma~\ref{L:connection_lemma} for the linear category $\bfA'_\calC(\bbGamma)$ is impossible, in general. Indeed, using the terminology introduced in Section~\ref{S:admissible_cobordisms}, admissibility is equivalent to strong admissibility for all objects and morphisms of $\bfA_\calC(\bbGamma)$, while this is no longer true for objects and morphisms of $\bfA'_\calC(\bbGamma)$. For this reason, the analogue of Lemma~\ref{L:connection_lemma} for the linear category $\bfA'_\calC(\bbGamma)$ requires stronger hypotheses.

\begin{lemma}\label{L:connection_lemma_prime}
 If $\bbSigma'_{\bbGamma} = (\Sigma'_\Gamma,P',\lambda')$ and $\bbSigma'''_{\bbGamma} = (\Sigma'''_\Gamma,P''',\lambda''')$ are strong\-ly ad\-mis\-si\-ble ob\-jects of $\bfA'_\calC(\bbGamma)$, and if $M'$ is a non-emp\-ty connected $3$-di\-men\-sion\-al cobordism with corners from $\Sigma'_\Gamma$ to $\Sigma'''_\Gamma$, then the linear map 
 \[
  \begin{array}{rccc}
   \pi_{M'} : & \adSk(M';P',P''') & 
   \to & \bfA'_\calC(\bbGamma)(\bbSigma'_{\bbGamma},\bbSigma'''_{\bbGamma}) \\
   & T' & \mapsto & [M',T',0]
  \end{array}
 \]
 is surjective.
\end{lemma}

The exact same proof of Lemma~\ref{L:connection_lemma} can be used to establish  Lemma~\ref{L:connection_lemma_prime}. Now, if $\Sigma$ is a $2$-dimensional cobordism, we denote with $\adD(\Sigma)$ the set of all sets of marked points in $\Sigma$ whose restriction to every connected component of $\Sigma$ is admissible. Then, the next result is stated in terms of the notion of dominating set, as recalled in Section~\ref{S:complete_lin_cat}.

\begin{lemma}\label{L:Morita_reduction}
 If $\bbGamma = \Gamma$ is an object of $\bfadCob_\calC$, if $\Sigma$ is a non-empty $2$-di\-men\-sion\-al cobordism from $\varnothing$ to $\Gamma$, and if $\lambda \subset H_1(\Sigma;\R)$ is a Lagrangian subspace, then the set
 \[
  \left\{ (\Sigma,P,\lambda) \in \bfA_\calC(\bbGamma) \bigm| P \in \adD(\Sigma) \right\}
 \]
 dominates $\bfA_\calC(\bbGamma)$.
\end{lemma}

\begin{proof}
 The proof is adapted from the one of \cite[Lemma~4.5]{D17}. If $\bbSigma_{\bbGamma} = (\Sigma_\Gamma,P,\lambda)$ and $\bbSigma''_{\bbGamma} = (\Sigma''_\Gamma,P'',\lambda'')$ are objects of $\bfA_\calC(\bbGamma)$, if $M$ is a non-empty connected $3$-di\-men\-sion\-al cobordism with corners from $\Sigma_\Gamma$ to $\Sigma$, and if $M'$ is a non-empty connected $3$-di\-men\-sion\-al cobordism with corners from $\Sigma$ to $\Sigma''_\Gamma$, then, thanks to Lemma~\ref{L:connection_lemma}, every morphism $[\bbM_{\bbGamma}]$ of $\bfA_\calC(\bbGamma)(\bbSigma_{\bbGamma},\bbSigma''_{\bbGamma})$ is the image of some vector of $\adSk(M \cup_\Sigma M';P,P'')$. Up to isotopy, every bichrome graph in $M \cup_\Sigma M'$ determines a set $P$ of marked points in $\Sigma$ providing an element of $\adD(\Sigma)$. This means we have
 \[
  \left[ \bbM_{\bbGamma} \right] = \sum_{i=1}^m \alpha_i \cdot \left[ (M',T'_i,0) \ast (M,T_i,0) \right],
 \]
 where $T_i \in \adSk(M;P,P_i)$ and $T'_i \in \adSk(M';P_i,P'')$ for every integer $1 \leqslant i \leqslant m$ and for some $P_1,\ldots,P_m \in \adD(\Sigma)$. 
\end{proof}

Notice that Lemma~\ref{L:Morita_reduction} does not admit an analogue for $\bfA'_\calC(\bbGamma)$, the reason being again that morphisms of $\bfA'_\calC(\bbGamma)$ are not strongly admissible in general.

\begin{lemma}\label{L:triviality_lemma}
 If $\bbSigma_{\bbGamma} = (\Sigma_\Gamma,P,\lambda)$ and $\bbSigma''_{\bbGamma} = (\Sigma''_\Gamma,P'',\lambda'')$ are objects of $\bfA_\calC(\bbGamma)$ for some object $\bbGamma = \Gamma$ of $\bfadCob_\calC$, if $\Sigma'$ is a non-empty $2$-di\-men\-sion\-al cobordism from $\Gamma$ to $\varnothing$, if $\lambda' \subset H_1(\Sigma';\R)$ is a Lagrangian subspace, and if $M$ and $M'$ are connected $3$-di\-men\-sion\-al cobordisms from $\varnothing$ to $\Sigma_\Gamma \cup_\Gamma \Sigma'$ and from $\Sigma''_\Gamma \cup_\Gamma \Sigma'$ to $\varnothing$ respectively, then a linear combination of morphisms
 \[
  \sum_{i=1}^m \alpha_i \cdot [ \bbM_{\bbGamma,i} ] \in \bfA_\calC(\bbGamma)(\bbSigma_{\bbGamma},\bbSigma''_{\bbGamma})
 \]
 is trivial if and only if
 \[
  \sum_{i=1}^m \alpha_i \rmL'_\calC \left( (M',T',0) \ast \left( \id_{(\Sigma',P',\lambda')}  \circ \bbM_{\bbGamma,i} \right) \ast (M,T,0) \right) = 0
 \]
 for all $P' \in \adD(\Sigma')$, $T \in \adSk(M;\varnothing,P' \cup P)$, and $T' \in \adSk(M';P' \circ P'',\varnothing)$.
\end{lemma}

\begin{proof}
 The proof follows closely the one for \cite[Lemma~4.6]{D17}. The morphism
 \[
  \sum_{i=1}^m \alpha_i \cdot [ \bbM_{\bbGamma,i} ] \in \bfA_\calC(\bbGamma)(\bbSigma_{\bbGamma},\bbSigma''_{\bbGamma})
 \]
 is zero, by definition, if and only if the linear map
 \[
  \sum_{i=1}^m \alpha_i \cdot \rmV_\calC \left( \id_{\bbSigma'_{\bbGamma}} \circ \bbM_{\bbGamma,i} \right)
 \]
 is zero for every object $\bbSigma'_{\bbGamma} \in \bfA'_\calC(\bbGamma)$. This happens if and only if the invariant
 \[
  \sum_{i=1}^m \alpha_i \rmL'_\calC \left( \bbM'_{\bbSigma'_{\bbGamma} \circ \bbSigma''_{\bbGamma}} \ast \left( \id_{\bbSigma'_{\bbGamma}} \circ \bbM_{\bbGamma,i} \right) \ast \bbM_{\bbSigma'_{\bbGamma} \circ \bbSigma_{\bbGamma}} \right)
 \]
 is zero for all vectors $[\bbM_{\bbSigma'_{\bbGamma} \circ \bbSigma_{\bbGamma}}] \in \rmV_\calC(\bbSigma'_{\bbGamma} \circ \bbSigma_{\bbGamma})$ and $[\bbM'_{\bbSigma'_{\bbGamma} \circ \bbSigma''_{\bbGamma}}] \in \rmV'_\calC(\bbSigma'_{\bbGamma} \circ \bbSigma''_{\bbGamma})$. The bichrome graph appearing in $\bbM_{\bbSigma'_{\bbGamma} \circ \bbSigma_{\bbGamma}}$ contains a blue edge labeled by a projective object of $\calC$ in every connected component of the underlying $3$-di\-men\-sion\-al cobordism. Therefore, up to isotoping these blue edges through all connected components of the $3$-di\-men\-sion\-al cobordism underlying $\id_{\bbSigma'_{\bbGamma}}$, we can suppose $\id_{\bbSigma'_{\bbGamma}}$ is strongly admissible. Then, thanks to Lemma~\ref{L:connection_lemma_prime}, we know that $[\id_{\bbSigma'_{\bbGamma}}]$ is the image of a vector of $\adSk(M'' \cup_{\Sigma'} M''';P',P')$ for some pair of non-empty connected $3$-di\-men\-sion\-al cobordisms with corners $M''$ from $\Sigma'_\Gamma$ to $\Sigma'$ and $M'''$ from $\Sigma'$ to $\Sigma'''_\Gamma$. Up to isotopy and red-to-blue operations, every bichrome graph in $M'' \cup_{\Sigma'} M'''$ determines a set of marked points in $\Sigma'$ which provides an element of $\adD(\Sigma')$. This means we have
 \[
  \left[ \id_{\bbSigma'_{\bbGamma}} \right] = \sum_{i'=1}^{m'} \alpha'_{i'} \cdot \left[ (M''',T'''_{i'},0) \ast (M'',T''_{i'},0) \right],
 \]
 where $T''_{i'} \in \adSk(M'';P',P'_{i'})$ and $T'''_{i'} \in \Sk(M''';P'_{i'},P')$ for every integer $1 \leqs i' \leqs m'$ and for some $P'_1,\ldots,P'_{m'} \in \adD(\Sigma')$. Now we can apply Lemma~\ref{L:connection_lemma} to
 \[
  \left[ \left( (M'',T''_{i'},0) \circ \id_{\bbSigma_{\bbGamma}} \right) \ast \bbM_{\bbSigma'_{\bbGamma} \circ \bbSigma_{\bbGamma}} \right]
 \]
 and to
 \[
  \left[ \bbM'_{\bbSigma'_{\bbGamma} \circ \bbSigma''_{\bbGamma}} \ast \left( (M''',T'''_{i'},0) \circ \id_{\bbSigma''_{\bbGamma}} \right) \right]
 \]
 for every integer $1 \leqs i' \leqs m'$ in order to conclude.
\end{proof}

\section{Monoidality}\label{S:monoidality}

In this section, which runs parallel to \cite[Section~6.1]{D17}, we establish our main result: the $2$-func\-tor $\hat{\bfA}_\calC : \bfadCob_\calC \rightarrow \coCat_\Bbbk$ of Section~\ref{S:extended_universal_construction} is symmetric monoidal in the sense of \cite[Definition~D.16]{D17}, and thus provides an ETQFT. To do this, we start by introducing some notation. Let $\rmI(\calC)$ be a set of representatives of isomorphism classes of simple objects of $\calC$ containing $\one$, and let us fix the projective generator
\[
 G = \bigoplus_{V \in \rmI(\calC)} P_V
\]
of $\calC$, where $P_V$ denotes the projective cover of $V$. We decompose the identity morphism of $G$ as
\[
 \id_G = \sum_{V \in \rmI(\calC)} \iota_{P_V} \circ \pi_{P_V},
\]
where $\pi_{P_V} : G \to P_V$ and $\iota_{P_V} : P_V \to G$ are the projection and injection morphisms corresponding to the direct summand $P_V$. The \textit{$G$-labeled sphere} is defined to be the $1$-mor\-phism $\bbS^2_{(+,G)} : \varnothing \rightarrow \varnothing$ of $\bfadCob_\calC$ given by 
\[
 \left( S^2,P_{(+,G)},\{ 0 \} \right),
\]
where $P_{(+,G)} \subset S^2$ is given by a single marked point (the north pole) with orientation $+$ and label $G$. Its adjoint $1$-mor\-phism $\overline{\bbS^2_{(-,G)}} : \varnothing \rightarrow \varnothing$ in $\bfadCob_\calC$ is given by
\[
 \left( \overline{S^2},P_{(-,G)},\{ 0 \} \right),
\]
where $P_{(-,G)} \subset \overline{S^2}$ is obtained from $P_{(+,G)}$ by reversing its orientation. The counit for this adjunction is the $2$-mor\-phism $(\bbD^1 \times \bbS^2)_G : \id_\varnothing \Rightarrow \overline{\bbS^2_{(-,G)}} \sqcup \bbS^2_{(+,G)}$ of $\bfadCob_\calC$ given by
\[
 (D^1 \times S^2,D^1 \times P_{(+,G)},0)
\]
for the blue tangle $D^1 \times P_{(+,G)}$ from $\varnothing$ to $P_{(-,G)} \sqcup P_{(+,G)}$ of Figure~\ref{F:D_1_times_S_2}.

\begin{figure}[hb]\label{F:D_1_times_S_2}
 \centering
 \includegraphics{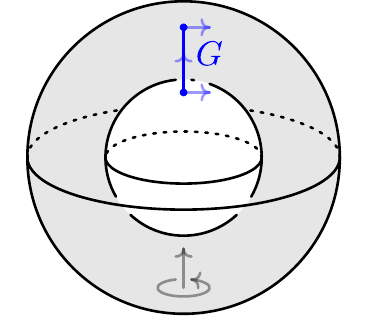}
 \caption{The $2$-morphism $(\bbD^1 \times \bbS^2)_G$ of $\bfadCob_\calC$.}
\end{figure}

Next, we consider the $2$-mor\-phisms $\overline{\bbD^3_{\pi_{\one}}} : \id_\varnothing \Rightarrow \overline{\bbS^2_{(-,G)}}$ and $\bbD^3_{\iota_{\one}} :  \id_\varnothing \Rightarrow \bbS^2_{(+,G)}$ of $\bfadCob_\calC$ given by
\begin{align*}
 \overline{\bbD^3_{\pi_{\one}}} &= \left( \overline{D^3},T_{\pi_{\one}},0 \right), 
 &
 \bbD^3_{\iota_{\one}} &= \left( D^3,T_{\iota_{\one}},0 \right)
\end{align*}
for the blue graphs $T_{\pi_{\one}} \subset \overline{D^3}$ from $\varnothing$ to $P_{(-,G)}$ and $T_{\eta_{\one}} \subset D^3$ from $\varnothing$ to $P_{(+,G)}$ of Figure~\ref{F:S_0_times_D_3}, where the morphisms $\pi_{\one} : G \to \one$ and $\iota_{\one} : \one \to G$ are defined by $\pi_{\one} = \epsilon_{\one} \circ \pi_{P_{\one}}$ and $\iota_{\one} = \iota_{P_{\one}} \circ \eta_{\one}$ for the projective cover morphism $\epsilon_{\one} : P_{\one} \to \one$ and the injective envelope morphism $\eta_{\one} : \one \to P_{\one}$. Notice that, when representing bichrome graphs embedded inside $3$-di\-men\-sion\-al cobordisms, we specify orientations for vertical boundaries of blue coupons. More specifically, arrows are directed from bottom base to top base, and thus prescribe the correct way of reading labels.

\begin{figure}[htb]\label{F:S_0_times_D_3}
 \centering
 \includegraphics{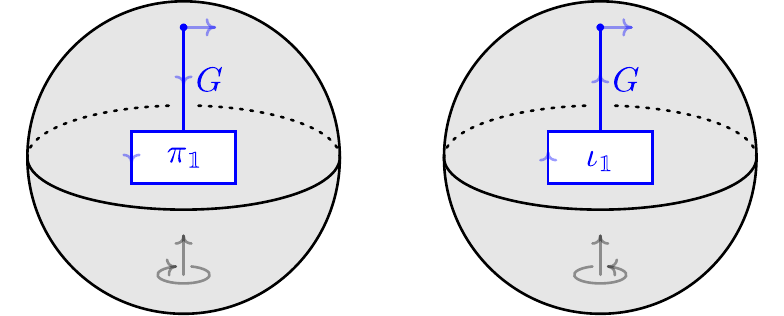}
 \caption{The $2$-morphisms $\overline{\bbD^3_{\pi_{\one}}}$ and $\bbD^3_{\iota_{\one}}$ of $\bfadCob_\calC$.}
\end{figure}

\begin{lemma}\label{L:cutting}
 For every object $\bbGamma = \Gamma$ of $\bfadCob_\calC$ and every object $\bbSigma_{\bbGamma} = (\Sigma_\Gamma,P,\lambda)$ of $\bfA_\calC(\bbGamma)$ we have
 \[
  \left[ (\bbD^1 \times \bbS^2)_G \disjun \id_{\bbSigma_{\bbGamma}} \right] = \calD \cdot  \left[ \overline{\bbD^3_{\pi_{\one}}} \disjun \bbD^3_{\iota_{\one}} \disjun \id_{\bbSigma_{\bbGamma}} \right] \in \bfA_\calC (\bbGamma) \left( \bbSigma_{\bbGamma},\overline{\bbS^2_{(-,G)}} \disjun \bbS^2_{(+,G)} \disjun \bbSigma_{\bbGamma} \right)
 \] 
 for the non-zero coefficient $\calD \in \Bbbk$ fixed in Section~\ref{S:3-manifold_invariant}.
\end{lemma}

\begin{proof}
 The strategy is to use Lemma~\ref{L:triviality_lemma}, so let us fix a non-empty $2$-di\-men\-sion\-al cobordism $\Sigma'$ from $\Gamma$ to $\varnothing$ and connected $3$-di\-men\-sion\-al cobordisms $M$ and $M'$ from $\varnothing$ to $\Sigma_\Gamma \cup_\Gamma \Sigma'$ and from $\overline{S^2} \sqcup S^2 \sqcup (\Sigma_\Gamma \cup_\Gamma \Sigma')$ to $\varnothing$ respectively. In order to verify the equality, we have to consider all admissible sets of marked points $P' \in \adD(\Sigma')$ and all admissible bichrome graphs
 \begin{align*}
  T &\in \adSk \left( M;\varnothing,P \cup P' \right), &
  T' &\in \adSk \left( M';P_{(-,G)} \sqcup P_{(+,G)} \sqcup (P \cup P'),\varnothing \right).
 \end{align*}
 On one hand, if $L \subset S^3$ is a framed oriented link with $\ell$ components and signature $\sigma$ providing a surgery presentation for
 \[
  M \cup_{\Sigma_\Gamma \cup_\Gamma \Sigma'} \left( \overline{D^3} \sqcup D^3 \sqcup (\Sigma_\Gamma \times I) \right) \cup_{\overline{S^2} \sqcup S^2 \sqcup (\Sigma_\Gamma \cup_\Gamma \Sigma')} M',
 \]
 then, up to isotopy and admissible skein equivalence, the renormalized Lyubashenko invariant we need to compute is given by
 \[
  \calD^{-1-\ell} \delta^{-\sigma} F'_\intL \left( \pic{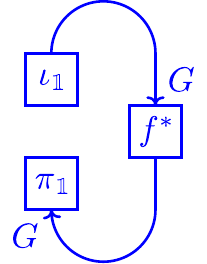} \right)
 \]
 for some endomorphism $f \in \End_\calC(G)$. On the other hand, a surgery presentation for
 \[
  M \cup_{\Sigma_\Gamma \cup_\Gamma \Sigma'} \left( (D^1 \times S^2) \sqcup (\Sigma_\Gamma \times I) \right) \cup_{\overline{S^2} \sqcup S^2 \sqcup (\Sigma_\Gamma \cup_\Gamma \Sigma')} M',
 \]
 is obtained from $L \subset S^3$ by adding a single red unknot with framing zero, and up to isotopy and admissible skein equivalence, the renormalized Lyubashenko invariant we need to compute is given by
 \[
  \calD^{-2-\ell} \delta^{-\sigma} F'_\intL \left( \pic{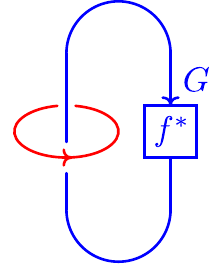} \right)
 \]
 for the same endomorphism $f \in \End_\calC(G)$ as above. Then the claim is a direct consequence of \cite[Lemma~4.4]{DGGPR19}.
\end{proof}

Our main result will be stated in terms of the c-equivalence relation introduced in Section~\ref{S:complete_lin_cat}. To do this, we first need to define the relevant coherence data for $\hat{\bfA}_\calC$, as in \cite[Definition~D.16]{D17}. We start with the definition of a linear functor
\[
 \bfepsilon : \Bbbk \rightarrow \bfA_\calC(\varnothing)
\]
sending the unique object of $\Bbbk$ to the object $\id_{\varnothing}$ of $\bfA_\calC(\varnothing)$.

\begin{proposition}\label{P:epsilon}
 The linear functor 
 \[ 
  \bfepsilon : \Bbbk \rightarrow \bfA_\calC(\varnothing)
 \]
 is a c-equivalence.
\end{proposition}

\begin{proof}
 Lemma~\ref{L:Morita_reduction} implies that the linear category $\bfA_\calC(\varnothing)$ is dominated by the family of objects 
 \[
  \{ (S^2,P,\{ 0 \}) \in \bfA_\calC(\varnothing) \mid P \in \adD(D^2) \}.
 \]
 For every object $(S^2,P,\{ 0 \})$ in this family, let $F_\calC(P) \in \Proj(\calC)$ denote the tensor product of all the labels of $P$ adjusted for orientation (a marked point with label $V$ contributes a factor of $V$ to the tensor product $F_\calC(P)$ if its orientation is positive, and it contributes a factor of $V^*$ otherwise). Since $G$ is a projective generator of $\calC$, there exist morphisms $f_1,\ldots,f_m \in \calC(F_\calC(P),G)$ and $f'_1,\ldots,f'_m \in \calC(G,F_\calC(P))$ satisfying
 \[
  \id_{F_\calC(P)} = \sum_{i=1}^m f'_i \circ f_i.
 \]
 Then, up to admissible skein equivalence, the object $\bbS^2_{(+,G)}$ dominates $\bfA_\calC(\varnothing)$. Now it is easy to show that, taking $\bbSigma_{\bbGamma} = \bbS^2_{(+,G)}$, Lemma~\ref{L:cutting} implies that also the object $\id_\varnothing$ dominates $\bfA_\calC(\varnothing)$. Since $\id_{\varnothing}$ is simple, the result follows immediately from Proposition~\ref{P:Morita_equivalence}.
\end{proof}

Next, we consider the $2$-trans\-for\-ma\-tion
\[
 \bfmu : \sqtimes \circ \left( \bfA_\calC \times \bfA_\calC \right) \Rightarrow \bfA_\calC \circ \disjun
\]
of Section~\ref{S:extended_universal_construction}. Recall that in Proposition~\ref{P:lax_monoidality_ETQFT} we proved $\bfmu_{\bbGamma,\bbGamma'}$ is faithful for all objects $\bbGamma, \bbGamma' \in \bfadCob_\calC$. We will now prove that this functor is also essentially surjective and full, thus promoting it to a c-equivalence.

\begin{proposition}\label{P:mu}
 For all objects $\bbGamma$ and $\bbGamma'$ of $\bfadCob_\calC$
 \[
  \bfmu_{\bbGamma,\bbGamma'} : \bfA_\calC(\bbGamma) \sqtimes \bfA_\calC(\bbGamma') \rightarrow \bfA_\calC(\bbGamma \disjun \bbGamma')
 \]
 is a c-equivalence.
\end{proposition}

\begin{proof}
 Again, the strategy is to use Proposition~\ref{P:Morita_equivalence}. First of all, if $\bfA_\calC(\bbGamma)$ is dominated by 
 \[
  \left\{ (\Sigma,P,\lambda) \in \bfA_\calC(\bbGamma) \bigm| P \in \adD(\Sigma) \right\}
 \]
 for some fixed $\Sigma$ and $\lambda$, and if $\bfA_\calC(\bbGamma')$ is dominated by 
 \[
  \left\{ (\Sigma',P',\lambda') \in \bfA_\calC(\bbGamma') \Bigm| P' \in \adD(\Sigma') \right\}
 \]
 for some fixed $\Sigma'$ and $\lambda'$, then $\bfA_\calC(\bbGamma \disjun \bbGamma')$ is dominated by
 \[
  \left\{ (\Sigma \sqcup \Sigma',P \sqcup P',\lambda \oplus \lambda') \in \bfA_\calC(\bbGamma \disjun \bbGamma') \Bigm| P \in \adD(\Sigma), P' \in \adD(\Sigma') \right\},
 \]
 as follows from Lemma~\ref{L:Morita_reduction}. Therefore, $\bfmu_{\bbGamma,\bbGamma'}$ defines a bijection between dominating sets. Furthermore, we already showed in Proposition~\ref{P:lax_monoidality_ETQFT} that $\bfmu_{\bbGamma,\bbGamma'}$ is faithful. This means we only need to show that $\bfmu_{\bbGamma,\bbGamma'}$ is full. To see it, we need to show that, for all objects 
 \begin{align*}
  \bbSigma_{\bbGamma} &= (\Sigma_\Gamma,P,\lambda)
  \in \bfA_\calC(\bbGamma), &
  \bbSigma''_{\bbGamma} &= (\Sigma''_\Gamma,P'',\lambda'') 
  \in \bfA_\calC(\bbGamma), \\*
  \bbSigma_{\bbGamma'} &= (\Sigma_{\Gamma'},P',\lambda')
  \in \bfA_\calC(\bbGamma'), &
  \bbSigma''_{\bbGamma'} &= (\Sigma''_{\Gamma'},P''',\lambda''')
  \in \bfA_\calC(\bbGamma'),
 \end{align*}
 every morphism
 \[
  [ \bbM_{\bbGamma \disjun \bbGamma'} ] \in \bfA_\calC(\bbGamma \disjun \bbGamma') \left( \bbSigma_{\bbGamma} \disjun \bbSigma_{\bbGamma'}, \bbSigma''_{\bbGamma} \disjun \bbSigma''_{\bbGamma'} \right)
 \]
 can be written as
 \[
  \sum_{i=1}^m \alpha_i \cdot \left[\bbM_{\bbGamma,i} \disjun \bbM_{\bbGamma',i} \right]
 \]
 for some scalars $\alpha_i \in \Bbbk$ and some morphisms 
 \begin{align*}
  [ \bbM_{\bbGamma,i} ] &\in \bfA_\calC(\bbGamma)(\bbSigma_{\bbGamma},\bbSigma''_{\bbGamma}), 
  &
  [ \bbM_{\bbGamma',i} ] &\in \bfA_\calC(\bbGamma')(\bbSigma_{\bbGamma'},\bbSigma''_{\bbGamma'}).
 \end{align*}
 In order to do so, let us consider non-empty connected $3$-di\-men\-sion\-al cobordisms with corners $M$ from $\Sigma_\Gamma \sqcup \overline{S^2}$ to $\Sigma''_\Gamma$ and $M'$ from $S^2 \sqcup \Sigma_{\Gamma'}$ to $\Sigma''_{\Gamma'}$, which determine the cobordism with corners
 \[
  M \cup_{S^2} M' := \left( ( \Sigma_\Gamma \times I ) \sqcup (D^1 \times S^2) \sqcup ( \Sigma_{\Gamma'} \times I ) \right) \cup_{\left( \Sigma_\Gamma \sqcup \overline{S^2} \sqcup S^2 \sqcup \Sigma_{\Gamma'} \right)} (M \sqcup M')
 \]
 from $\Sigma_\Gamma \sqcup \Sigma_{\Gamma'}$ to $\Sigma''_\Gamma \sqcup \Sigma''_{\Gamma'}$. Then, since $M \cup_{S^2} M'$ is non-empty and connected, Lemma~\ref{L:connection_lemma} implies the vector space $\bfA_\calC(\bbGamma \disjun \bbGamma') \left( \bbSigma_{\bbGamma} \disjun \bbSigma_{\bbGamma'}, \bbSigma''_{\bbGamma} \disjun \bbSigma''_{\bbGamma'} \right)$ is generated by vectors of the form
 \[
  \left[ M \cup_{S^2} M',T,0 \right]
 \]
 for all admissible bichrome graphs $T$ inside $M \cup_{S^2} M'$ from $P \sqcup P'$ to $P'' \sqcup P'''$. Let us choose such an admissible bichrome graph $T$, and let us show that the corresponding morphism of $\bfA_\calC(\bbGamma \disjun \bbGamma') \left( \bbSigma_{\bbGamma} \disjun \bbSigma_{\bbGamma'}, \bbSigma''_{\bbGamma} \disjun \bbSigma''_{\bbGamma'} \right)$ lies in the image of $\mu_{\bbGamma,\bbGamma'}$. Since $T$ is admissible, we can suppose, up to red-to-blue operations, that $D^1 \times S^2$ intersects only blue edges of $T$. Furthermore, up to isotopy, we can suppose that $D^1 \times S^2$ intersects a projective blue edge of $T$. Then, up to admissible skein equivalence, using the fact that $\Proj(\calC)$ is an ideal, we can suppose that $D^1 \times S^2$ intersects a single blue edge whose color $V$ is a projective object of $\calC$. Now, since $G$ is a projective generator of $\calC$, there exist morphisms $f_1,\ldots,f_m \in \calC(V,G)$ and $f'_1,\ldots,f'_m \in \calC(G,V)$ satisfying
 \[
  \id_V = \sum_{i=1}^m f'_i \circ f_i.
 \]
 This means we can decompose $[M \cup_{S^2} M',T,0]$ as
 \[
  \sum_{i=1}^m \left[ \left( (M,T_{f_i},0) \disjun {}(M,T_{f'_i},0) \right) \ast \left( \id_{\bbSigma_{\bbGamma}} \disjun {}(\bbD^1 \times \bbS^2)_G \disjun \id_{\bbSigma_{\bbGamma'}} \right) \right].
 \]
 for some bichrome graphs $T_{f_i}$ inside $M$ from $P \sqcup P_{(-,G)}$ to $P''$ and $T_{f'_i}$ inside $M$ from $P_{(+,G)} \sqcup P'$ to $P'''$ induced by $f_i$ and by $f'_i$ respectively. Then, Lemma~\ref{L:cutting} implies
 \begin{align*}
  &[M \cup_{S^2} M',T,0] \\*
  &\hspace*{\parindent} = \sum_{i=1}^n \left[ \left( (M,T_{f_i},0) \disjun {}(M,T_{f'_i},0) \right) \ast \left( \id_{\bbSigma_{\bbGamma}} \disjun {}(\bbD^1 \times \bbS^2)_G \disjun \id_{\bbSigma_{\bbGamma'}} \right) \right] \\*
  &\hspace*{\parindent} = \sum_{i=1}^n \calD \cdot \left[ \left( (M,T_{f_i},0) \disjun {}(M,T_{f'_i},0) \right) \ast \left( \id_{\bbSigma_{\bbGamma}} \disjun \left( \overline{\bbD^3_{\pi_{\one}}} \disjun \bbD^3_{\iota_{\one}} \right) \disjun \id_{\bbSigma_{\bbGamma'}} \right) \right] \\*
  &\hspace*{\parindent} = \sum_{i=1}^n \calD \cdot \left[ \left( (M_\Sigma,T_{f_i},0) \circ \left( \id_{\bbSigma_{\bbGamma}} \disjun \overline{\bbD^3_{\pi_{\one}}} \right) \right) \disjun 
  \left( (M_{\Sigma'},T_{f'_i},0) \circ \left( \bbD^3_{\iota_{\one}} \disjun \id_{\bbSigma_{\bbGamma'}} \right) \right) \right].
 \end{align*}
\end{proof}

Let us now prove our main result.

\begin{theorem}\label{T:symmetric_monoidality}
 The $2$-functor $\hat{\bfA}_\calC : \bfadCob_\calC \rightarrow \coCat_\Bbbk$ is symmetric monoidal.
\end{theorem}

\begin{proof}
 The claim follows directly from the results we have established up to here, but in order to prove it we need to check carefully that all the conditions required by \cite[Definition~D.16]{D17} are met. First of all, the completion $\cmpl \circ \bfepsilon$ of the linear functor $\bfepsilon$ is an equivalence, where $\cmpl$ denotes the $2$-func\-tor of Proposition~\ref{P:co_strict_2-funct}. Indeed, this follows immediately from Proposition~\ref{P:epsilon}. Next, thanks to Proposition~\ref{P:mu}, the $2$-trans\-for\-ma\-tion $( \cmpl \triangleright \bfmu ) \circ ( \bfchi \triangleleft (\bfF \times \bfF))$ is a composition of equivalences, where the $2$-trans\-for\-ma\-tion $\bfchi : \cmpl \circ {\sqtimes} \circ (\cmpl \times \cmpl) \Rightarrow \cmpl \circ \sqtimes$ is part of the coherence data of the symmetric monoidal $2$-func\-tor $\cmpl$ given by Proposition~\ref{P:completion} (see \cite[Definition~D.6]{D17} for a definition of the operation of left whiskering appearing here). Furthermore, we claim
 \begin{gather*}
  \bfmu_{\varnothing,\bbGamma} \circ (\bfepsilon \sqtimes \id_{\bfA_\calC(\bbGamma)}) = \id_{\bfA_\calC(\bbGamma)}, \quad \bfmu_{\bbGamma,\varnothing} \ast (\id_{\bfA_\calC(\bbGamma)} \sqtimes \bfepsilon) = \id_{\bfA_\calC(\bbGamma)}, \\
  \bfmu_{\bbGamma,\bbGamma' \disjun \bbGamma''} \circ \left( \id_{\bfA_\calC(\bbGamma)} \sqtimes \bfmu_{\bbGamma',\bbGamma''} \right) = \bfmu_{\bbGamma \disjun \bbGamma',\bbGamma''} \ast \left( \bfmu_{\bbGamma,\bbGamma'} \sqtimes \id_{\bfA_\calC(\bbGamma'')} \right)
 \end{gather*}
 for all $\bbGamma, \bbGamma', \bbGamma'' \in \bfadCob_{\calC}$. Indeed, this is an immediate consequence of the quasi-strictness of $\bfadCob_{\calC}$. This means that we can assemble all the coherence data for $\hat{\bfA}_\calC$ simply from identity $2$-mod\-i\-fi\-ca\-tions. Therefore, all the conditions we need to check are trivially satisfied.
\end{proof}

\section{Identification of the image}\label{S:identification_of_image}

In this section, which is analogue to \cite[Chapter~7]{D17}, we characterize the ETQFT $\hat{\bfA}_\calC$ by means of the modular category $\calC$. First, we identify the linear category associated to the generating object of $\bfadCob_{\calC}$, the circle. Next, we study linear functors associated to generating $1$-morphisms of $\bfadCob_\calC$, discs and pants.

\subsection{Circle category}\label{S:circle}

We start by describing the \textit{circle category} of $\hat{\bfA}_\calC$, which is the linear category $\hat{\bfA}_\calC(\bbS^1)$ for the generating object $\bbS^1 = S^1$ of $\bfadCob_\calC$. We will show that $\hat{\bfA}_\calC(\bbS^1)$ is equivalent to the category $\Proj(\calC)$ of projective objects of $\calC$.

\begin{definition}\label{D:disc}
 The \textit{$V$-la\-beled disc} $\bbD^2_{(+,V)} : \varnothing \rightarrow \bbS^1$ is the $1$-mor\-phism of $\bfCob_\calC$ given by 
 \[
  \left( D^2,P_{(+,V)},\{ 0 \} \right)
 \]
 where $P_{(+,V)} \subset D^2$ is given by a single marked point (the center) with orientation $+$ and label $V$. Its adjoint $1$-mor\-phism $\overline{\bbD^2_{(-,V)}} : \bbS^1 \rightarrow \varnothing$ in $\bfadCob_\calC$ is given by 
 \[
  \left( \overline{D^2},P_{(-,V)},\{ 0 \} \right)
 \]
 where $P_{(-,V)} \subset \overline{D^2}$ is obtained from $P_{(+,V)}$ by reversing its orientation.
\end{definition}

Let
\[
 \rmF : \Proj(\calC) \rightarrow \bfA_\calC(\bbS^1)
\]
be the linear functor sending every object $V$ of $\Proj(\calC)$ to the object $\bbD^2_{(+,V)}$ of $\bfA_\calC(\bbS^1)$, and every morphism $f$ of $\calC(V,V')$ to the morphism $[ (\bbD^2 \times \bbI)_f ]$ of $\bfA_\calC(\bbS^1)(\bbD^2_{(+,V)},\bbD^2_{(+,V')})$, where the $2$-mor\-phism $(\bbD^2 \times \bbI)_f : \bbD^2_{(+,V)} \Rightarrow \bbD^2_{(+,V')}$ of $\bfadCob_\calC$ is given by
\[
 (D^2 \times I,T_f,0)
\]
for the blue graph $T_f \subset D^2 \times I$ from $P_{(+,V)}$ to $P_{(+,V')}$ of Figure~\ref{F:D_2_times_I}.

\begin{figure}[hbt]\label{F:D_2_times_I}
 \centering
 \includegraphics{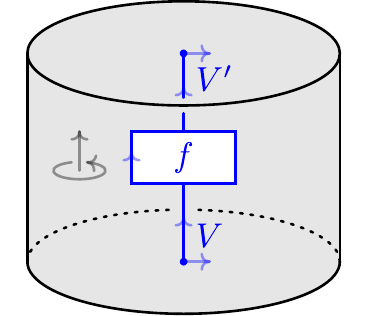}
 \caption{The $2$-morphism $(\bbD^2 \times \bbI)_f$ of $\bfadCob_\calC$.}
\end{figure}

\begin{proposition}\label{P:univ_lin_cat}
 The linear functor
 \[ 
  \rmF : \Proj(\calC) \rightarrow \bfA_\calC(\bbS^1)
 \]
 is a c-equivalence.
\end{proposition}

\begin{proof}
 First of all, functoriality of $\rmF$ follows directly from an admissible skein equivalence. Then, the strategy is to use Proposition~\ref{P:Morita_equivalence} to show that $\rmF$ is a c-equivalence. First of all, thanks to Lemma~\ref{L:Morita_reduction}, the family of objects 
 \[
  \{ (D^2,P,\{ 0 \}) \in \bfA_\calC(\bbS^1) \mid P \in \adD(D^2) \}
 \]
 dominates the linear category $\bfA_\calC(\bbS^1)$. However, it is easy to see that, up to admissible skein equivalence, the identity of an object $(D^2,P,\{ 0 \})$ in this family factors through the $F_\calC(P)$-la\-beled disc $\bbD^2_{(+,F_\calC(P))}$, where $F_\calC(P) \in \Proj(\calC)$ denotes the tensor product of all the labels of $P$ adjusted for orientation, just like in the proof of Lemma~\ref{L:cutting}. This means that also the family of objects 
 \[
  \{ \bbD^2_{(+,V)} \in \bfA_\calC(\bbS^1) \mid V \in \Proj(\calC) \}
 \]
 dominates $\bfA_\calC(\bbS^1)$. Therefore, we simply need to show that $\rmF$ is faithful and full. In order to do so, let us fix objects $V,V' \in \Proj(\calC)$, and let us consider the linear map
 \[
  \rmF_{V,V'} : \calC(V,V') \to \bfA_\calC(\bbS^1) (\bbD^2_{(+,V)},\bbD^2_{(+,V')}).
 \] 
 For what concerns fullness, surjectivity of $\rmF_{V,V'}$ follows directly from Lemma~\ref{L:connection_lemma} using admissible skein equivalence. For what concerns faithfulness, injectivity of $\rmF_{V,V'}$ follows from the non-degeneracy of the modified trace $\rmt$ on $\Proj(\calC)$. Indeed, by definition, this means that for every non-trivial morphism $f \in \calC(V,V')$ there exists some morphism $f' \in \calC(V',V)$ satisfying $\rmt_V(f' \circ f) \neq 0$. In particular, this determines a $2$-mor\-phism $\overline{(\bbD^2 \times \bbI)_{f'}} : \overline{\bbD^2_{(-,V)}} \Rightarrow \overline{\bbD^2_{(-,V')}}$ of $\bfadCob_\calC$ given by
 \[
  (\overline{D^2 \times I},T^*_{f'},0 )
 \]
 for the blue graph $T^*_{f'} \subset \overline{D^2 \times I}$ from $P_{(-,V)}$ to $P_{(-,V')}$ of Figure~\ref{F:D_2_times_I_pairing}.

 \begin{figure}[hbt]\label{F:D_2_times_I_pairing}
  \centering
  \includegraphics{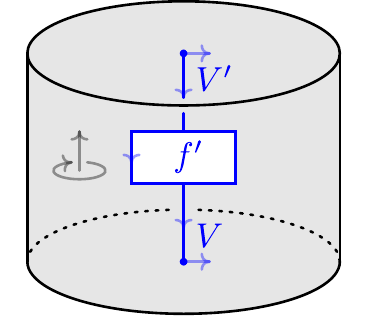}
  \caption{The $2$-morphism $\overline{(\bbD^2 \times \bbI)_{f'}}$ of $\bfadCob_\calC$.}
 \end{figure}

 Then we have
 \[
  \rmL'_\calC \left( \overline{\bbD^3_{\id_{V'}}} \ast \left( (\overline{\bbD^2 \times \bbI)_{f'}} \circ (\bbD^2 \times \bbI)_f \right) \ast \bbD^3_{\id_V} \right) = \calD^{-1} \rmt(f' \circ f) \neq 0,
 \]
 where the $2$-mor\-phisms
 \begin{align*}
  \bbD^3_{\id_V} : \id_\varnothing &\Rightarrow (\overline{\bbD^2_{(-,V)}} \circ \bbD^2_{(+,V)}) &
  \overline{\bbD^3_{\id_{V'}}} : (\overline{\bbD^2_{(-,V')}} \circ \bbD^2_{(+,V')}) &\Rightarrow \id_\varnothing
 \end{align*}
 of $\bfadCob_\calC$ are represented in Figure~\ref{F:D_3}, for $f = \id_V$, and in Figure~\ref{F:D_3_pairing}, for $f' = \id_{V'}$, respectively.
\end{proof}

\subsection{Disc functors}\label{S:discs}

We move on to describe linear functors induced by generating $1$-mor\-phisms of $\bfadCob_\calC$ under $\hat{\bfA}_\calC$ in terms of various structures supported by $\Proj(\calC)$. We start by focusing on discs, which are of two kinds. The first family of disc functors can be described in terms of projective objects of $\calC$. Indeed, for every object $V \in \Proj(\calC)$, we have a constant linear functor from $\Bbbk$ to $\bbProj(\calC)$, which we still denote $V$.

\begin{proposition}\label{P:unit}
 For every object $V \in \calC$ we have a commutative diagram of linear functors
 \begin{center}
 \begin{tikzpicture}[descr/.style={fill=white}]
  \node (P0) at ({atan(-0.5)}:{3.25/(2*sin(atan(0.5)))}) {$\bfA_\calC(\bbS^1)$};
  \node (P1) at (90:{3.25/2}) {$\bbProj(\calC)$};
  \node (P2) at ({atan(-0.5) + 2*90}:{3.25/(2*sin(atan(0.5)))}) {$\Bbbk$};
  \node (P3) at ({atan(0.5) + 2*90}:{3.25/(2*sin(atan(0.5)))}) {$\bfA_\calC(\varnothing)$};
  \draw
  (P1) edge[->] node[right,xshift=10pt] {$\rmF$} (P0)
  (P2) edge[->] node[above,yshift=5pt] {$V$} (P1)
  (P2) edge[->] node[left,xshift=-5pt] {$\bfepsilon$} (P3)
  (P3) edge[->] node[below,yshift=-5pt] {$\bfA_\calC \left( \bbD^2_{(+,V)} \right)$} (P0);
 \end{tikzpicture}
 \end{center}
\end{proposition}

\begin{proof}
 The claim obviously follows from the definitions.
\end{proof}

Next, we consider a family of disc functors which can be described in terms of vector spaces of morphism in $\calC$. For every object $V \in \calC$, not necessarily projective, let
\[
 \calC(V,\_) : \Proj(\calC) \rightarrow \Vect_\Bbbk
\]
be the linear functor sending every object $V'$ of $\Proj(\calC)$ to the vector space of morphisms $\calC(V,V')$, and every morphism $f'$ of $\calC(V',V'')$ to the linear map $\calC(V,f') : \calC(V,V') \rightarrow \calC(V,V'')$ sending every morphism $f$ of $\calC(V,V')$ to the morphism $f' \circ f$ of $\calC(V,V'')$. Next, let us consider the $3$-dimensional cobordism $D^3$ from $\varnothing$ to $D^2 \cup_{S^1} \overline{D^2}$. 

Then, for every object $V' \in \Proj(\calC)$, we consider the linear map
\[
 (\eta_V)_{V'} : \calC(V,V') \rightarrow \rmV_\calC \left( \overline{\bbD^2_{(-,V)}} \circ \bbD^2_{(+,V')} \right)
\]
sending every vector $f$ of $\calC(V,V')$ to the vector $[\bbD^3_f]$ of $\rmV_\calC(\overline{\bbD^2_{(-,V)}} \circ \bbD^2_{(+,V')})$, where the $2$-morphism $\bbD^3_f : \id_{\varnothing} \Rightarrow \overline{\bbD^2_{(-,V)}} \circ \bbD^2_{(+,V')}$ of $\bfadCob_\calC$ is given by
\[
 (D^3,T_f,0)
\]
for the blue graph $T_f \subset D^3$ from $\varnothing$ to $P_{(+,V')} \cup P_{(-,V)}$ of Figure~\ref{F:D_3}.

\begin{figure}[hbt]\label{F:D_3}
 \centering
 \includegraphics{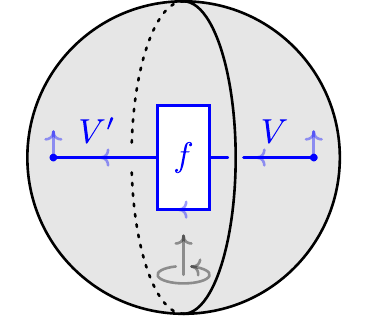}
 \caption{The $2$-morphism $\bbD^3_f$ of $\bfadCob_\calC$.}
\end{figure}

\newpage

\begin{proposition}\label{P:counit}
 For every object $V \in \calC$ the linear maps $(\eta_V)_{V'}$ define a natural isomorphism
 \begin{center}
  \begin{tikzpicture}[descr/.style={fill=white}]
   \node (P0) at ({atan(-0.5)}:{3.25/(2*sin(atan(0.5)))}) {$\bfA_\calC(\varnothing)$};
   \node (P1) at ({atan(0.5)}:{3.25/(2*sin(atan(0.5)))}) {$\Vect_\Bbbk$};
   \node (P2) at ({atan(-0.5) + 2*90}:{3.25/(2*sin(atan(0.5)))}) {$\Proj(\calC)$};
   \node (P3) at (3*90:{3.25/2}) {$\bfA_\calC(\bbS^1)$};
   \node[above] at (0,0) {$\Downarrow$};
   \node[below] at (0,0) {$\eta_V$};
   \draw
   (P0) edge[->] node[right,xshift=5pt] {$\rmV_\calC$} (P1)
   (P2) edge[->] node[above,yshift=5pt] {$\calC(V,\_)$} (P1)
   (P2) edge[->] node[left,xshift=-10pt] {$\rmF$} (P3)
   (P3) edge[->] node[below,yshift=-5pt] {$\bfA_\calC \left( \overline{\bbD^2_{(-,V)}} \right)$} (P0);
  \end{tikzpicture}
 \end{center}
\end{proposition}

\begin{proof}
 The result is established by showing that the linear maps
 \[
  (\eta_V)_{V'} : \calC(V,V') \rightarrow \rmV_\calC \left( \overline{\bbD^2_{(-,V)}} \circ \bbD^2_{(+,V')} \right)
 \]
 are natural with respect to morphisms $f' \in \calC(V',V'')$, and that they are invertible for all objects $V' \in \Proj(\calC)$. In order to show naturality, we need to check that
 \[
  \rmV_\calC \left( \left[ \id_{\overline{\bbD^2_{(-,V)}}} \circ 
  (\bbD^2 \times \bbI)_{f'} \right] \right)
  \left( \left[ \bbD^3_f \right] \right) \\
  = \left[ \bbD^3_{f' \circ f} \right]
 \]
 for every morphism $f \in \calC(V,V')$ and every morphism $f' \in \calC(V',V'')$. This follows directly from an admissible skein equivalence, exactly like the functoriality of $\rmF$ in the proof of Proposition~\ref{P:univ_lin_cat}. Next, for what concerns surjectivity of $(\eta_V)_{V'}$, we know from \cite[Proposition~4.11]{DGGPR19}, or equivalently from Lemma~\ref{L:connection_lemma}, that every vector of $\rmV_\calC(\overline{\bbD^2_{(-,V)}} \circ \bbD^2_{(+,V')})$ is the image of an admissible bichrome graph in $\adSk(D^3;\varnothing,P_{(+,V')} \cup P_{(-,V)})$. Furthermore, up to isotopy and admissible skein equivalence, we can always restrict to admissible bichrome graphs of the form $T_f$, as in the definition of $\bbD^3_f$. For what concerns injectivity of $(\eta_V)_{V'}$, we use the non-de\-gen\-er\-a\-cy of the modified trace $\rmt$: indeed, if we consider some non-trivial morphism $f \in \calC(V,V')$, then, since $\rmt$ is non-degenerate, there exists a morphism $f' \in \calC(V',V)$ satisfying $\rmt_V(f' \circ f) \neq 0$.
In particular, this determines the $2$-morphism $\overline{\bbD^3_{f'}} : \overline{\bbD^2_{(-,V)}} \circ \bbD^2_{(+,V')} \Rightarrow \id_{\varnothing}$ of $\bfadCob_\calC$ given by 
 \[
  (\overline{D^3},T^*_{f'},0)
 \]
 for the blue graph $T^*_{f'} \subset \overline{D^3}$ from $P_{(+,V')} \cup P_{(-,V)}$ to $\varnothing$ of Figure~\ref{F:D_3_pairing}.

\begin{figure}[hbt]\label{F:D_3_pairing}
 \centering
 \includegraphics{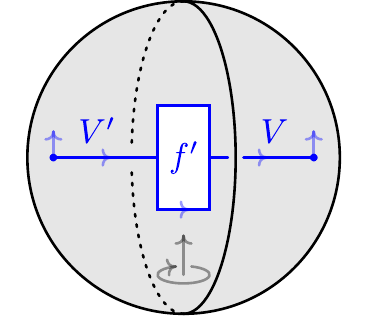}
 \caption{The $2$-morphism $\overline{\bbD^3_{f'}}$ of $\bfadCob_\calC$.}
\end{figure}

  This implies
 \[
  \rmL'_\calC \left( \overline{\bbD^3_{f'}} \ast \bbD^3_f \right) = \calD^{-1} \rmt(f' \circ f) \neq 0. \qedhere
 \]
\end{proof}

\subsection{Pant functors}\label{S:2-pants}

We finish by discussing a new family of $1$-morphisms of $\bfadCob_\calC$. We define the \textit{pant cobordism $P^2$} as the $2$-dimensional cobordism from $S^1 \sqcup S^1$ to $S^1$ given by $D^2$ minus two open balls of radius $\frac{1}{4}$ and center $(-\frac{1}{2},0)$ and $(+\frac{1}{2},0)$ respectively, with incoming and outgoing horizontal boundary identifications induced by $\id_{S^1}$ through rescaling and translation.

\begin{definition}\label{D:pant}
 The \textit{pant} $\bbP^2 : \bbS^1 \disjun \bbS^1 \rightarrow \bbS^1$ is the $1$-mor\-phism of $\bfadCob_\calC$ given by 
 \[
  \left( P^2,\varnothing,H_1(P^2;\R) \right).
 \]
 Its adjoint $1$-mor\-phism $\overline{\bbP^2} : \bbS^1 \rightarrow \bbS^1 \disjun \bbS^1$ in $\bfadCob_\calC$ is given by 
 \[
  \left( \overline{P^2},\varnothing,H_1(P^2;\R) \right).
 \]
\end{definition}

The first pant functor we are going to consider can be described in terms of the tensor product $\otimes$ on $\calC$. Let
\[
 \pant : \Proj(\calC) \sqtimes \Proj(\calC) \rightarrow \Proj(\calC)
\]
be the linear functor sending every object $(V,V')$ of 
$\Proj(\calC) \sqtimes \Proj(\calC)$ to the object
$V \otimes V'$ of $\Proj(\calC)$, and every morphism
$f \sqtimes f'$ of $\calC(V,V'') \sqtimes \calC(V',V''')$ to the morphism $f \otimes f'$ of $\calC(V \otimes V',V'' \otimes V''')$. 

Next, for every object $(V,V') \in \Proj(\calC) \sqtimes \Proj(\calC)$, we consider the morphism
\[
 \eta_{(V,V')} \in \bfA_\calC(\bbS^1) \left( \bbD^2_{(+,V \otimes V')},\bbP^2 \circ \left( \bbD^2_{(+,V)} \disjun \bbD^2_{(+,V')} \right) \right)
\]
given by $[(\bbD^2 \times \bbI)_{T_{V,V'}}]$, where the $2$-morphism of $\bfadCob_\calC$
\[
 (\bbD^2 \times \bbI)_{T_{(V,V')}} : \bbD^2_{(+,V \otimes V')} \Rightarrow \bbP^2 \circ \left( \bbD^2_{(+,V)} \disjun \bbD^2_{(+,V')} \right)
\]
is given by $(D^2 \times I,T_{(V,V')},0)$ for the blue graph $T_{(V,V')} \subset D^2 \times I$ from $P_{(+,V \otimes V')}$ to $P_{(+,V)} \cup P_{(+,V')}$ of Figure~\ref{F:D_2_times_I_P_2}.

\begin{figure}[hbt]\label{F:D_2_times_I_P_2}
 \centering
 \includegraphics{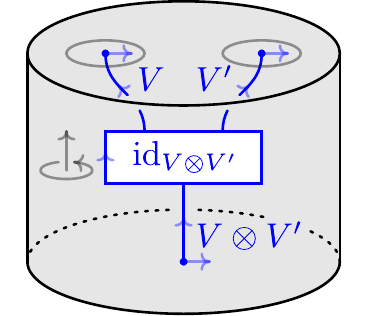}
 \caption{The $2$-morphism $(\bbD^2 \times \bbI)_{T_{(V,V')}}$ of $\bfadCob_\calC$.}
\end{figure}

\begin{proposition}\label{P:product}
 The morphisms $\eta_{(V,V')}$ define a natural isomorphism
 \begin{center}
  \begin{tikzpicture}[descr/.style={fill=white}]
   \node (P0) at ({atan(-0.5)}:{3.25/(2*sin(atan(0.5)))}) {$\bfA_\calC(\bbS^1)$};
   \node (P1) at (90:{3.25/2}) {$\Proj(\calC)$};
   \node (P2) at ({atan(-0.5) + 2*90}:{3.25/(2*sin(atan(0.5)))}) {$\Proj(\calC) \sqtimes \Proj(\calC)$};
   \node (P3) at (3*90:{3.25/2}) {$\bfA_\calC(\bbS^1 \disjun \bbS^1)$};
   \node[above] at (0,0) {$\Downarrow$};
   \node[below] at (0,0) {$\eta$};
   \draw
   (P1) edge[->] node[right,xshift=10pt] {$\rmF$} (P0)
   (P2) edge[->] node[above,yshift=5pt] {$\pant$} (P1)
   (P2) edge[->] node[left,xshift=-10pt] {$\bfmu_{\bbS^1,\bbS^1} \circ 
   (\rmF \sqtimes \rmF)$} (P3)
   (P3) edge[->] node[below,yshift=-5pt] {$\bfA_\calC \left( \bbP^2 \right)$} (P0);
  \end{tikzpicture}
 \end{center}
\end{proposition}

\begin{proof}
 The result is established by showing that the morphisms
 \[
  \eta_{(V,V')} \in \bfA_\calC(\bbS^1) \left( \bbD^2_{(+,V \otimes V')},\bbP^2 \circ \left( \bbD^2_{(+,V)} \disjun \bbD^2_{(+,V')} \right) \right)
 \]
 are natural with respect to morphisms $f \in \calC(V,V'')$ and $f' \in \calC(V',V''')$, and that they are invertible for all objects $V \in \Proj(\calC)$ and $V' \in \Proj(\calC)$. For what concerns naturality, we need to show that
 \begin{align*}
  &\left( \bfA_\calC(\bbP^2) \left( \left[ (\bbD^2 \times \bbI)_f \disjun{} (\bbD^2 \times \bbI)_{f'} \right] \right) \right) \circ \left[ (\bbD^2 \times \bbI)_{T_{(V,V')}} \right] \\
  &\hspace{\parindent} = \left[ (\bbD^2 \times \bbI)_{T_{(V'',V''')}} \right] \circ \left[ (\bbD^2 \times \bbI)_{f \otimes f'} \right]
 \end{align*}
 for every morphism $f \sqtimes f' \in \calC(V,V'') \sqtimes \calC(V',V''')$. This follows directly from an admissible skein equivalence. For what concerns invertibility of $\eta_{(V,V')}$, the inverse of $\eta_{(V,V')}$ is given by $[(\bbD^2 \times \bbI)_{T_{V,V'}^{-1}}]$, where the $2$-morphism of $\bfadCob_\calC$
\[
 (\bbD^2 \times \bbI)_{T_{V,V'}^{-1}} : \bbP^2 \circ \left( \bbD^2_{(+,V)} \disjun \bbD^2_{(+,V')} \right) \Rightarrow \bbD^2_{(+,V \otimes V')}
\]
is simply given by $(D^2 \times I,T_{(V,V')}^{-1},0)$ for the blue graph $T_{(V,V')}^{-1} \subset D^2 \times I$ from $P_{(+,V)} \cup P_{(+,V')}$ to $P_{(+,V \otimes V')}$ of Figure~\ref{F:D_2_times_I_P_2_inverse}.
\end{proof}

\begin{figure}[hbt]\label{F:D_2_times_I_P_2_inverse}
 \centering
 \includegraphics{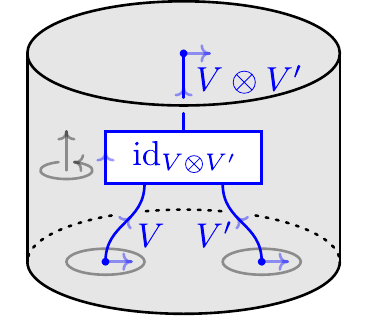}
 \caption{The $2$-morphism $(\bbD^2 \times \bbI)_{T_{(V,V')}^{-1}}$ of $\bfadCob_\calC$.}
\end{figure}

We move on to study the adjoint pant functor, which can be described in terms of the coend $\coend$ in $\calC$. In order to do this, we first need some preparation, so let us start by fixing the projective generator $G$ of $\calC$ considered in Section~\ref{S:monoidality}. For all objects $X \in \calC$ and $V \in \Proj(\calC)$ we denote with $i_{X,V} \in \calC(X^* \otimes X \otimes V,G^* \otimes G \otimes V)$ the morphism defined by
\[
 i_{X,V} := \zeta^{-1} \cdot F_\Lambda \left( \pic{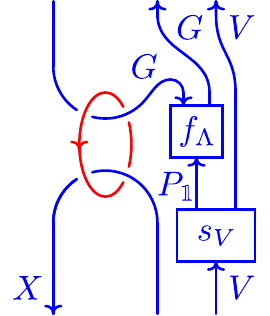} \right)
\]
for the non-zero coefficient $\zeta = \calD^2 = \Delta_+ \Delta_- \in \Bbbk$ fixed in Section~\ref{S:3-manifold_invariant}, and for the pair of morphisms $f_\Lambda \in \calC(P_{\one},G^* \otimes G)$ and $s_V \in \calC(V,P_{\one} \otimes V)$ satisfying 
\begin{align*}
 i_G \circ f_\Lambda &= \Lambda \circ \epsilon_{\one}, 
 &
 (\epsilon_{\one} \otimes \id_V) \circ s_V &= \id_V,
\end{align*}
whose existence is showed in \cite[Lemma~4.5]{DGGPR19}.

\begin{proposition}\label{P:uccidetemi}
 The morphism 
 \[
  \ell_V := i_{G,V} \in \End_\calC(G^* \otimes G \otimes V)
 \]
 is idempotent, and the object
 \[
  \im \ell_V \in \Proj(\calC),
 \]
 together with the dinatural family of morphisms 
 \[
  \{ i_{X,V} : X^* \otimes X \otimes V \to \im \ell_V \mid X \in \calC \},
 \]
 satisfies
 \[
  \im \ell_V = \int^{X \in \calC} X^* \otimes X \otimes V,
 \]
 which means it is isomorphic to $\coend \otimes V$, which is also the coend of the functor
 \begin{align*}
  (\_^* \otimes \_ \otimes V) : \calC^\op \times \calC & \to \Proj(\calC) \\*
  (X,Y) & \mapsto X^* \otimes Y \otimes V.
 \end{align*}
 Furthermore, $\ell_V$ belongs to the image of the functor 
 \[
  \pant : \Proj(\calC) \sqtimes \Proj(\calC) \to \Proj(\calC),
 \]
 which means there exists a morphism
 \[
  \myuline{\ell_V} = \sum_{i=1}^n \ell'_i \sqtimes \ell''_{V,i} \in \End_\calC(G^*) \sqtimes \End_\calC(G \otimes V)
 \]
 satisfying
 \[
  \ell_V = \sum_{i=1}^n \ell'_i \otimes \ell''_{V,i} \in \End_\calC(G^* \otimes G \otimes V).
 \]
\end{proposition}

The proof of Proposition~\ref{P:uccidetemi} is postponed to Appendix~\ref{A:projective_coend}. Now let
\[
 \pantbar : \Proj(\calC) \rightarrow \Proj(\calC) \sqtimes \Proj(\calC)
\]
be the linear functor sending every object $V$ of $\Proj(\calC)$ to the object
\[
 \myuline{\coend \otimes V} := \im \myuline{\ell_V} = \im \left( \sum_{i=1}^n \ell'_i \sqtimes \ell''_{V,i} \right)
\]
of $\Proj(\calC) \sqtimes \Proj(\calC)$ and every morphism $f$ of $\calC(V,V')$ to the morphism
\[
 \myuline{\ell_{V'}} \circ \left( \id_{G^*} \sqtimes \left( \id_G \otimes f \right) \right) \circ \myuline{\ell_V} = \sum_{i,j=1}^n \left( \ell'_j \circ \ell'_i \right) \sqtimes \left( \ell''_{V',j} \circ f \circ \ell''_{V,i} \right)
\]
of
\[
 \left( \calC \sqtimes \calC \right) \left( \myuline{\coend \otimes V}, \myuline{\coend \otimes V'} \right).
\]

Next, we consider a $3$-di\-men\-sion\-al cobordism with corners $W$ from $D^2 \sqcup D^2$ to $D^2 \cup_{S^1} \overline{P^2}$ given by $I \times D^2 \subset \R^3$ minus two open balls of radius $\frac{1}{4}$ and centers $(0,0,0)$ and $(1,0,0)$. Then, for every object $V \in \Proj(\calC)$, we consider the morphism
\[
 \eta_V \in \bfA_\calC(\bbS^1 \disjun \bbS^1) \left( \im \left( \sum_{i=1}^n \left[ \left( \bbD^2 \times \bbI \right)_{\ell'_i} \disjun \left(\bbD^2 \times \bbI \right)_{\ell''_{V,i}} \right] \right), \overline{\bbP^2} \circ \bbD^2_{(+,V)} \right)
\]
given by $[ \bbW_V ]$, where the $2$-morphism of $\bfadCob_\calC$
\[
 \bbW_V : \bbD^2_{(+,G^*)} \disjun \bbD^2_{(+,G \otimes V)} \Rightarrow \overline{\bbP^2} \circ \bbD^2_{(+,V)}
\]
is given by $(W,T_V,0)$ for the blue graph $T_V \subset W$ from $P_{(+,G^*)} \disjun P_{(+,G \otimes V)}$ to $P_{(+,V)}$ represented in Figure~\ref{F:I_times_D_2}.

\begin{figure}[hbt]\label{F:I_times_D_2}
 \centering
 \includegraphics{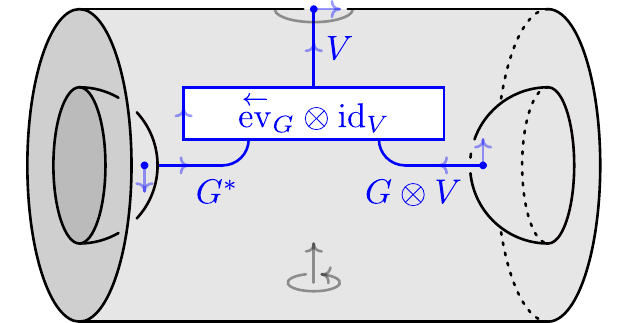}
 \caption{The $2$-morphism $\bbW_V$ of $\bfadCob_\calC$.} 
\end{figure}

\newpage

\begin{proposition}\label{P:coproduct_critical}
 The morphisms $\eta_V$ define a natural isomorphism
 \begin{center}
  \begin{tikzpicture}[descr/.style={fill=white}]
   \node (P0) at ({atan(-0.5)}:{3.25/(2*sin(atan(0.5)))}) {$\bfA_\calC(\bbS^1 \disjun \bbS^1)$};
   \node (P1) at (90:{3.25/2}) {$\Proj(\calC) \sqtimes \Proj(\calC)$};
   \node (P2) at ({atan(-0.5) + 2*90}:{3.25/(2*sin(atan(0.5)))}) {$\Proj(\calC)$};
   \node (P3) at (3*90:{3.25/2}) {$\bfA_\calC(\bbS^1)$};
   \node[above] at (0,0) {$\Downarrow$};
   \node[below] at (0,0) {$\eta$};
   \draw
   (P1) edge[->] node[right,xshift=10pt] {$\bfmu_{\bbS^1,\bbS^1} \circ (\rmF \sqtimes \rmF)$} (P0)
   (P2) edge[->] node[above,yshift=5pt] {$\pantbar$} (P1)
   (P2) edge[->] node[left,xshift=-10pt] {$\rmF$} (P3)
   (P3) edge[->] node[below,yshift=-5pt] {$\bfA_\calC \left( \overline{\bbP^2} \right)$} (P0);
  \end{tikzpicture}
 \end{center}
\end{proposition}

\begin{proof}
 The result is established by showing that the morphisms
 \[
  \eta_V \in \bfA_\calC(\bbS^1 \disjun \bbS^1) \left( \im \left( \sum_{i=1}^n \left[ \left( \bbD^2 \times \bbI \right)_{\ell'_i} \disjun \left(\bbD^2 \times \bbI \right)_{\ell''_{V,i}} \right] \right), \overline{\bbP^2} \circ \bbD^2_{(+,V)} \right)
 \]
 are natural with respect to morphisms $f \in \calC(V,V')$, and that they are invertible for all objects $V \in \Proj(\calC)$. For what concerns naturality, we need to show that
 \begin{align*}
  &\sum_{i,j=1}^n \left[ \bbW_{V'} \right] \circ \left[ \left( \bbD^2 \times \bbI \right)_{\ell'_j \circ \ell'_i} \disjun \left( \bbD^2 \times \bbI \right)_{\ell''_{V',j} \circ f \circ \ell''_{V,i}} \right] 
  = \left[ \id_{\overline{\bbP^2}} \circ (\bbD^2 \times \bbI)_f \right] \circ \left[ \bbW_V \right]
 \end{align*}
 for every morphism $f \in \calC(V,V')$. This follows directly from an admissible skein equivalence. Next, for what concerns invertibility of $\eta_V$, we claim that the inverse of the morphism
 \[
  \left[ \bbW_V \right] \in 
  \bfA_\calC(\bbS^1 \disjun \bbS^1) \left( \im \left( \sum_{i=1}^n \left[ \left( \bbD^2 \times \bbI \right)_{\ell'_i} \disjun \left(\bbD^2 \times \bbI \right)_{\ell''_{V,i}} \right] \right), \overline{\bbP^2} \circ \bbD^2_{(+,V)} \right)
 \]
 is given by the morphism 
 \[
  \zeta^{-1} \cdot \left[ \overline{\bbW_V} \right] \in 
  \bfA_\calC(\bbS^1 \disjun \bbS^1) \left( \overline{\bbP^2} \circ \bbD^2_{(+,V)}, \im \left( \sum_{i=1}^n \left[ \left( \bbD^2 \times \bbI \right)_{\ell'_i} \disjun \left(\bbD^2 \times \bbI \right)_{\ell''_{V,i}} \right] \right) \right),
 \] 
 where the $2$-morphism of $\bfadCob_\calC$
\[
 \overline{\bbW_V} : \overline{\bbP^2} \circ \bbD^2_{(+,V)} \Rightarrow \bbD^2_{(+,G^*)} \disjun \bbD^2_{(+,G \otimes V)}
\]
 is given by $(\overline{W},T^*_V,0)$ for the blue graph $T^*_V \subset \overline{W}$ from $P_{(+,V)}$ to $P_{(+,G^*)} \disjun P_{(+,G \otimes V)}$ of Figure~\ref{F:I_times_D_2_inverse}.

 \begin{figure}[ht]\label{F:I_times_D_2_inverse}
  \centering
  \includegraphics{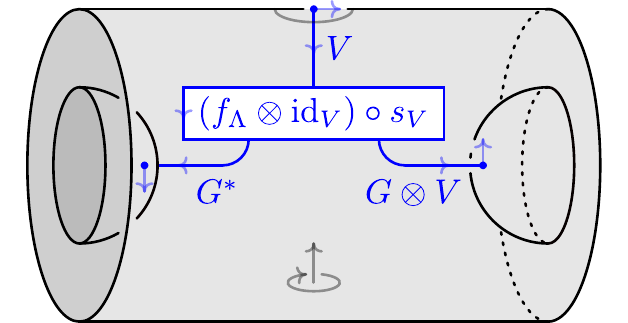}
  \caption{The $2$-morphism $\overline{\bbW_V}$ of $\bfadCob_\calC$.} 
 \end{figure}

 Indeed, on the one hand the equality
 \[
  \zeta^{-1} \cdot \left[ \overline{\bbW_V} \right] \circ \left[ \bbW_V \right] 
  = \sum_{i=1}^n \left[ \left( \bbD^2 \times \bbI \right)_{\ell'_i} \disjun \left(\bbD^2 \times \bbI \right)_{\ell''_{V,i}} \right]
 \]
 follows from Lemma~\ref{L:cutting} and Proposition~\ref{P:uccidetemi}. On the other hand, we have
 \begin{align*}
  \zeta^{-1} \cdot \left[ \bbW_V \right] \circ \left[ \overline{\bbW_V} \right]
  = \zeta^{-1} \cdot \left[  ( \bbI \times \bbS^1 \times \bbI )_V \right],
 \end{align*}
 where the $2$-morphism of $\bfadCob_\calC$
\[
 ( \bbI \times \bbS^1 \times \bbI )_V : \overline{\bbP^2} \circ \bbD^2_{(+,V)} \Rightarrow \overline{\bbP^2} \circ \bbD^2_{(+,V)}
\]
 is given by $(I \times S^1 \times I,T_V,0)$ for the bichrome graph $T_V \subset I \times S^1 \times I$ of Figure~\ref{F:I_times_S_1_times_I}. Thanks to \cite[Lemma~4.4]{DGGPR19}, this means
 \begin{align*}
  \zeta^{-1} \cdot \left[ \bbW_V \right] \circ \left[ \overline{\bbW_V} \right] 
  = \left[ \id_{\overline{\bbP^2} \circ \bbD^2_{(+,V)}} \right].
 \end{align*}
\end{proof}

 \begin{figure}[hbt]\label{F:I_times_S_1_times_I}
  \centering
  \includegraphics{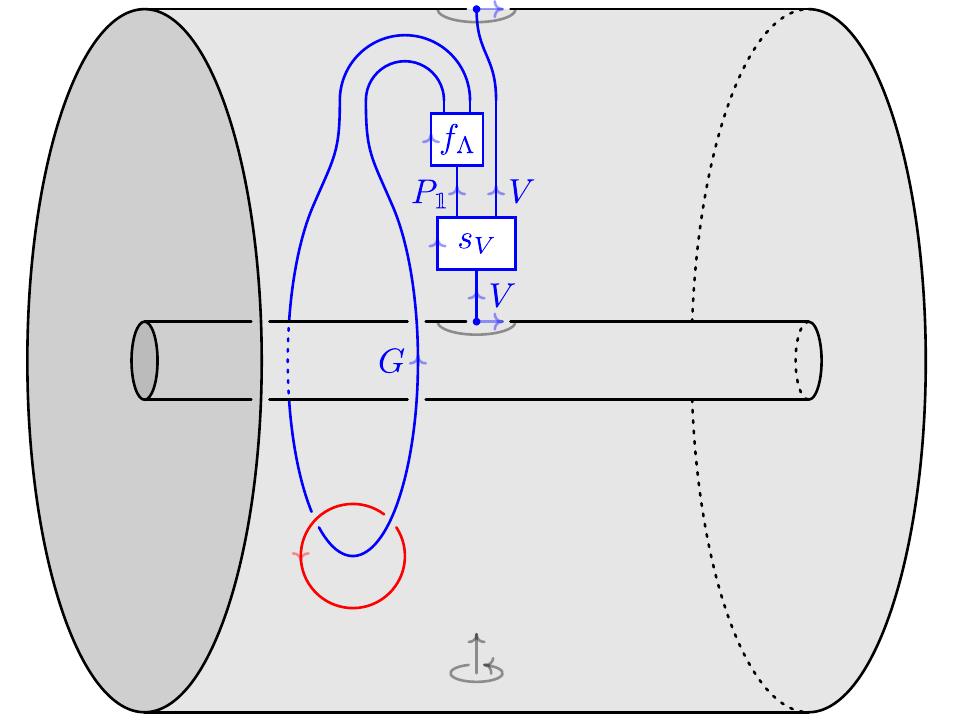}
  \caption{The $2$-mor\-phism $( \bbI \times \bbS^1 \times \bbI )_V$ of $\bfadCob_\calC$.}
 \end{figure}

\newpage

\appendix

\section{Projective coend}\label{A:projective_coend}

In this appendix, we prove Proposition~\ref{P:uccidetemi}.

\begin{proof}[Proof of Proposition~\ref{P:uccidetemi}]
 Let us start by showing that
 \[
  i_{G,V} \circ i_{X,V} = i_{X,V}
 \]
 for every $X \in \calC$. Indeed, we have
 \begin{align*}
  \pic[-15]{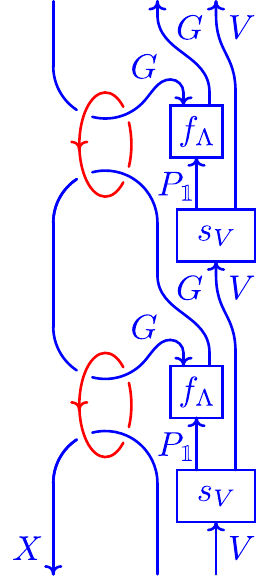} \hspace*{5pt}
  \doteq \hspace*{-7.5pt} \pic[-15]{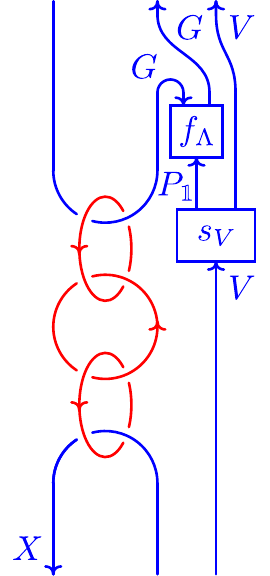} \hspace*{5pt}
  \doteq \hspace*{-7.5pt} \pic[-15]{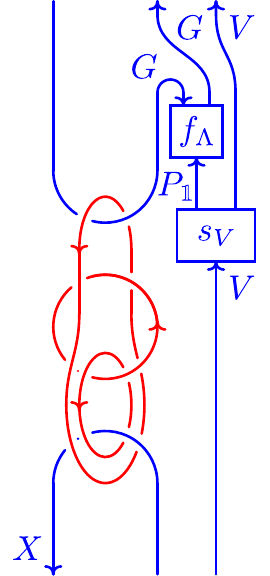} \hspace*{5pt}
  &\doteq \phantom{\zeta \cdot {}} \hspace*{-15pt} \pic[-15]{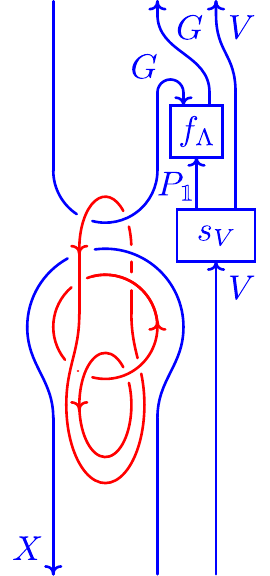} \\*[5pt]
  &\doteq \zeta \cdot \hspace*{-15pt} \pic{projective_coend}
 \end{align*}
 where the first skein equivalence follows from \cite[Lemma~4.5]{DGGPR19}, the second and third ones from \cite[Proposition~3.7]{DGGPR19}, and the fourth one from \cite[Lemma~2.7]{DGGPR19}. This has a few direct consequences. First of all, $\ell_V = i_{G,V}$ is idempotent, which is the first claim in the statement. Next, $i_{X,V}$ defines a morphism with target $\im \ell_V$ for every $X \in \Proj(\calC)$. Furthermore, the family
 \[
  \{ i_{X,V} : X^* \otimes X \otimes V \to \im \ell_V \mid X \in \calC \}
 \]
 is clearly dinatural. This means that $\im \ell_V$ is indeed a dinatural transformation with source $(\_^* \otimes \_ \otimes V) : \calC^\op \times \calC \to \Proj(\calC)$. Then, we need to prove it satisfies the universal property of the coend. In other words, we need to show that, for every object $Z \in \Proj(\calC)$ and every dinatural transformation $\eta : (\_^* \otimes \_ \otimes V) \din Z$, there exists a unique morphism $f_V(\eta) : \im \ell_V \to Z$ satisfying
 \[
  f_V(\eta) \circ i_{X,V} = \eta_X
 \]
 for all $X \in \calC$. We claim
 \[
  f_V(\eta) = \eta_G.
 \]
 Indeed, consider the monodromy transformation $\Omega : \coend \otimes \coend \to \coend \otimes \coend$, which is uniquely determined by the condition
 \[
  \Omega \circ (i_X \otimes i_Y) = (i_X \otimes i_Y) \circ (\id_{X^*} \otimes (c_{Y^*,X} \circ c_{X,Y^*}) \otimes \id_Y )
 \]
 for all $X,Y \in \calC$, and whose inverse $\Omega^{-1} : \coend \otimes \coend \to \coend \otimes \coend$ is uniquely determined by the condition
 \[
  \Omega \circ (i_X \otimes i_Y) = (i_X \otimes i_Y) \circ (\id_{X^*} \otimes (c_{X,Y^*}^{-1} \circ c_{Y^*,X}^{-1}) \otimes \id_Y )
 \]
 for all $X,Y \in \calC$. Then, the S-transformation $\calS : \coend \to \coend$ of \cite[Section~1.3]{L94} is defined as
 \[
  \calS := (\epsilon \otimes \id_\coend) \circ \Omega \circ (\id_\coend \otimes \intL)
 \]
 for the counit $\epsilon : \coend \to \one$ and the integral $\intL : \one \to \coend$ of $\coend$, and its inverse satisfies
 \[
  \calS^{-1} := (\id_\coend \otimes \epsilon) \circ \Omega^{-1} \circ (\intL \otimes \id_\coend).
 \]
 See \cite[Figure~3]{DGGPR20} for a graphical representation of these morphisms. Then, \cite[Lemma~4.5]{DGGPR19} implies
 \[
  \eta_G \circ i_{X,V} = f_\calC(\eta) \circ \left( \left( \calS \circ \calS^{-1} \circ i_X \right) \otimes \id_V \right) = f_\calC(\eta) \circ \left( i_X \otimes \id_V \right) = \eta_X,
 \]
 where $f_\calC(\eta) : \coend \otimes V \to Z$ is the unique morphism determined by the universal property of the coend
 \[
  \coend \otimes V = \int^{X \in \calC} X^* \otimes X \otimes V,
 \]
 which has structure morphisms $i_X \otimes \id_V : X^* \otimes X \otimes V \to \coend \otimes V$ for every $X \in \calC$. In particular, this proves that $\eta_G$ is indeed a morphism with source $\im \ell_V$, because, taking $X = G$, we get
 \[
  \eta_G \circ \ell_V = \eta_G \circ i_{G,V} = \eta_G.
 \]
 Now, uniqueness of $f_V(\eta)$ follows from the fact that if $\tilde{f}_V(\eta)  : \im \ell_V \to Z$ is another morphism satisfying
 \[
  \tilde{f}_V(\eta) \circ i_{X,V} = \eta_X
 \]
 for every $X \in \calC$, then the equality must hold for $X = G$ too, which means
 \[
  \eta_G = \tilde{f}_V(\eta) \circ i_{G,V} = \tilde{f}_V(\eta) \circ \ell_V = \tilde{f}_V(\eta),
 \]
 because the source of $\tilde{f}_V(\eta)$ is $\im \ell_V$.
 
 Finally, let us show that $\ell_V$ belongs to the image of the functor
 \[
  \pant : \Proj(\calC) \sqtimes \Proj(\calC) \to \Proj(\calC).
 \]
 Indeed, since $G$ is a projective generator of $\calC$, we know there exist morphisms $f_1,\ldots,f_m \in \calC(G \otimes G^*,G)$ and $f'_1,\ldots,f'_m \in \calC(G,G \otimes G^*)$ satisfying
 \[
  \id_{G \otimes G^*} = \sum_{i=1}^m f'_i \circ f_i.
 \]
 Then we have
 \[
  \pic{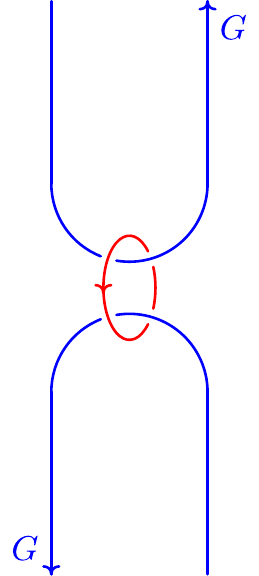}
  \doteq \sum_{i=1}^m \pic{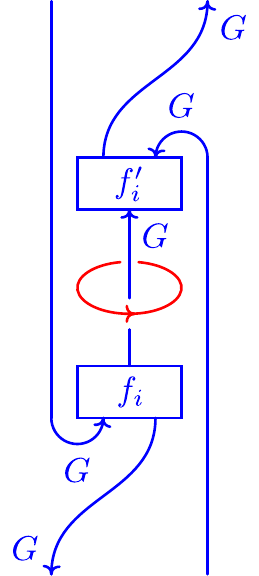}
  \doteq \sum_{i=1}^m \zeta \cdot \pic{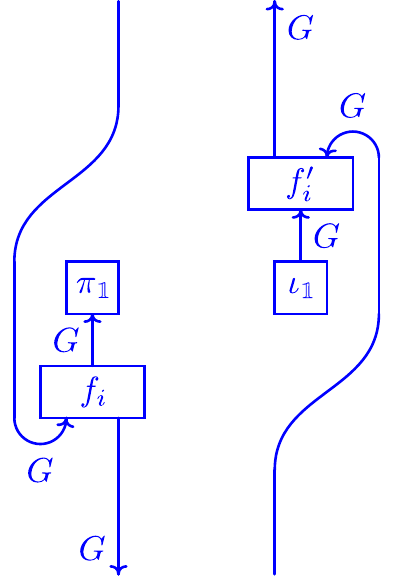}
 \]
 where the second skein equivalence follows from \cite[Lemma~4.4]{DGGPR19}.
\end{proof}

\addtocontents{toc}{\protect\setcounter{tocdepth}{2}}

\end{document}